\documentclass[11pt,reqno]{amsart}
\usepackage{graphics}
\usepackage{amssymb}
\usepackage{amsmath}
\usepackage{amsfonts}
\DeclareMathAlphabet\mathbfcal{OMS}{cmsy}{b}{n}
\usepackage [ruled] {algorithm2e} % with border lines
\usepackage{mathtools}
\usepackage{mathbbol}
\usepackage{todonotes}
\usepackage{centernot}
\usepackage{comment}
\usepackage{graphicx}
\usepackage{float}
\usepackage{subcaption}
\usepackage{algorithmic}
\usepackage{color,soul}
\usepackage{bm}
\usepackage{xcolor}
\usepackage{xcolor,colortbl}
\usepackage[mathlines]{lineno}

\usepackage[scale=0.65]{geometry} % Reduce document margins

\usepackage{hyperref}

\hypersetup{
colorlinks=true,
linkcolor=red,
filecolor=magenta,      
urlcolor=cyan,
citecolor=green
}
\usepackage[most]{tcolorbox}

\newtcolorbox{shadedcvbox}[1][]{enhanced jigsaw,
colback=white!80!blue,
coltext={black},
boxrule=0.5pt,
arc=3mm,
auto outer arc,
boxsep=3pt,
left=4pt,
right=2pt,
bottom=2pt,
top=2pt,
#1}

\newtheorem{theorem}{Theorem}[section]

\newtheorem{algo}[theorem]{Algorithm}
\newtheorem{lemma}[theorem]{Lemma}

\theoremstyle{definition}
\newtheorem{definition}[theorem]{Definition}
\theoremstyle{remark}
\newtheorem{remark}[theorem]{Remark}
\newtheorem{example}[theorem]{Example}

\newcommand{\RR}{\mathbb{R}}
\newcommand{\uu}{\mathbf{u}} % Vector in a Hilbert space
\newcommand{\vv}{\mathbf{v}} % Vector in a Hilbert space
\newcommand{\bs}{\mathbf{s}} % Vector in a Hilbert space
\newcommand{\by}{\mathbf{y}} % Vector in a Hilbert space
\newcommand{\bom}{\mathbf{m}} % Point in the Grassmanian
\newcommand{\bS}{\mathbf{S}} % Snapshot matrix
\newcommand{\bU}{\mathbf{U}} % matrix
\newcommand{\bV}{\mathbf{V}} % matrix
\newcommand{\bY}{\mathbf{Y}} % matrix
\newcommand{\bZ}{\mathbf{Z}} % matrix
\newcommand{\bM}{\mathbf{M}} % matrix
\newcommand{\bSig}{\boldsymbol{\Sigma}} % matrix
\newcommand{\bPhi}{\boldsymbol{\Phi}} % Spatial vectors
\newcommand{\bPsi}{\boldsymbol{\Psi}} % Spatial vectors

\newcommand{\wtlambda}{\widetilde{\lambda}}

\newcommand{\VecMat}[2]{\mathrm{Mat}_{#1,#2}(\RR)}
\newcommand{\VecMatG}[2]{\mathrm{Mat}^{0}_{#1,#2}(\RR)}
\newcommand{\StieComp}[2]{\mathcal{S}t_{c}(#1,#2)}

\DeclareMathOperator{\tr}{tr}
\DeclareMathOperator{\Exp}{Exp}
\DeclareMathOperator{\Log}{Log}

\begin{document}
%
%\linenumbers
\title[ST Interpolation of Parametrized Rigid-Viscoplastic FEM Problems]{A Non-Intrusive Space-Time Interpolation from Compact Stiefel Manifolds of Parametrized Rigid-Viscoplastic FEM Problems}
%\subtitle{Do you have a subtitle?\\ If so, write it here}

\author{O. Friderikos}
\address[Orestis Friderikos]{ Université Paris-Saclay, ENS Paris-Saclay, CNRS,  LMT - Laboratoire de Mécanique et Technologie, 91190, Gif-sur-Yvette, France}
\email{ofriderikos@ihu.gr}

\author{M. Olive}
\address[Marc Olive]{ Université Paris-Saclay, ENS Paris-Saclay, CNRS,  LMT - Laboratoire de Mécanique et Technologie, 91190, Gif-sur-Yvette, France}
\email{marc.olive@math.cnrs.fr}

\author{E. Baranger}
\address[Emmanuel Baranger]{ Université Paris-Saclay, ENS Paris-Saclay, CNRS,  LMT - Laboratoire de Mécanique et Technologie, 91190, Gif-sur-Yvette, France}
\email{emmanuel.baranger@ens-paris-saclay.fr}

\author{D. Sagris}
\address[Dimitris Sagris]{Mechanical Engineering Department, Laboratory of Manufacturing Technology \& Machine Tools, International Hellenic University, GR-62124 Serres Campus, Greece.}
\email{dsagris@ihu.gr}

\author{C. David}
\address[Constantine David]{Mechanical Engineering Department, Laboratory of Manufacturing Technology \& Machine Tools, International Hellenic University, GR-62124 Serres Campus, Greece.}
\email{david@ihu.gr}
% \thanks is optional -

%\titlerunning{ST Interpolation of Parametrized Rigid-Viscoplastic FEM Problems}

\maketitle
\setcounter{tocdepth}{1}
\tableofcontents
\begin{abstract}

This work aims to interpolate parametrized Reduced Order Model (ROM) basis constructed via the Proper Orthogonal Decomposition (POD) to derive a robust ROM of the system's dynamics for an unseen target parameter value. A novel non-intrusive Space-Time (ST) POD basis interpolation scheme is proposed, for which we define ROM spatial and temporal basis \emph{curves on compact Stiefel manifolds}. An interpolation is finally defined on a \emph{mixed part} encoded in a square matrix directly deduced using the space part, the singular values and the temporal part, to obtain an interpolated snapshot matrix, keeping track of accurate space and temporal eigenvectors. Moreover, in order to establish a well-defined curve on the compact Stiefel manifold, we introduce a new procedure, the so-called oriented SVD. Such an oriented SVD produces unique right and left eigenvectors for generic matrices, for which all singular values are distinct. It is important to notice that the ST POD basis interpolation does not require the construction and the subsequent solution of a reduced-order FEM model as classically is done. Hence it is avoiding the bottleneck of standard POD interpolation which is associated with the evaluation of the nonlinear terms of the Galerkin projection on the governing equations. As a proof of concept, the proposed method is demonstrated with the adaptation of rigid-thermoviscoplastic finite element ROMs applied to a typical nonlinear open forging metal forming process. Strong correlations of the ST POD models with respect to their associated high-fidelity FEM counterpart simulations are reported, highlighting its potential use for near real-time parametric simulations using off-line computed ROM POD databases.
  
\end{abstract}
\begin{table}[H]
	\begin{tabular}{p{3cm} p{4.7cm} p{6.8cm}}
		%\hline
		\multicolumn{3}{p{13.2cm}}{\textbf{Notations}} \\
		%\hline
		$\VecMat{n}{p}$ & Set of $n \times p$ matrices in $\mathbb{R}$ & \\
		\hline
		$\mathbf{I}_{p}$ &  Identity matrix in $\VecMat{p}{p}$ & \\
		\hline
		$[\by_1,\dots,\by_p]$ &  Matrix in $\VecMat{n}{p}$ & Matrix with column vectors $\by_i\in \RR^n$ \\
		\hline
		$\mathrm{O}(p)$ &  Orthogonal group on $\RR^{p}$ &  $\left\{ Q \in \VecMat{p}{p},\quad Q^T Q = \mathbf{I}_{p}  \right\}$\\
		\hline
		\hline
		$\mathcal{G}(p,n)$ & Grassmann manifold  & Set of $p$ linear subspaces in $\RR^n$ \\
		\hline
		$\pi^{-1}(\bom)$ & Fiber at $\bom \in \mathcal{G}(p,n)$ & 
		\hspace*{-0.25cm}
		\begin{tabular}{l}
		If $\by_1,\dotsc,\by_p$ is an orthonormal basis of $\bom$ \\
		$\pi^{-1}(\bom)=\left\{ \bY Q,\quad Q\in \mathrm{O}(p),\quad \bY=[\by_1,\dotsc,\by_p]\right\}$
		\end{tabular}
		\\
		\hline
		$\mathcal{T}_{\bom}:=\mathcal{T}_{\bom}\mathcal{G}(p,n)$ & Tangent space of $\mathcal{G}(p,n)$ at $\bom$ & 
		\hspace*{-0.25cm}
		\begin{tabular}{l}
			For $\bY\in \pi^{-1}(\bom)$, one model of $\mathcal{T}_{\bom}$ is \\
			$\left\{ \bZ\in \VecMat{n}{p},\quad \bZ^{T}\bY=0\right\}$
		\end{tabular}
		\\
		\hline
		\hline
		$\mathcal{S}t(p,n)$  &  Stiefel manifold & \hspace*{-0.25cm} \begin{tabular}{l} Set of ordered $p$-tuples independent\\
		vectors in $\RR^n$ \end{tabular} \\
		\hline
		$\StieComp{p}{n}$  &  Compact Stiefel manifold &  \hspace*{-0.25cm}  \begin{tabular}{l} Set of ordered $p$-tuples of orthonormal \\
			vectors in $\RR^n$ \\
		$\StieComp{p}{n}=\left\{ \bY\in \VecMat{n}{p},\quad \bY^{T}\bY=\mathbf{I}_{p}\right\}$ 
		\end{tabular}
		\\
		\hline
		$\text{Hor}_{\bY}$ & Horizontal space at $\bY$ & $\text{Hor}_{\bY}:=\left\{ \bZ\in \VecMat{n}{p},\quad \bZ^{T}\bY =\mathbf{0} \right\}$ \\
		\hline
		 $v\in \mathcal{T}_{\bom}$ & Velocity vector on the tangent plane $\mathcal{T}_{\bom}$ & Represented by a horizontal lift $\bZ\in\text{Hor}_{\bY}$, with $\bY\in \pi^{-1}(\bom)$ \\
		\hline
		$\bS^{(i)}$ & Snapshot matrix & $\bS^{(i)} \in \VecMat{n}{m}$ corresponding to parameter value $\lambda_{i}$\\	
		\hline
	\end{tabular}
	%\caption{Numerical algorithms.}
	%\label{table:Numerical_algorithms}
\end{table}

\section{Introduction}

Computational metal forming has been widely used in academic laboratories and the manufacturing industry over the last decades, becoming nowadays a mature, well established technology. Nevertheless, new challenging fields are emerging, among others, uncertainty quantification, optimization of processes and parameter identification in design analysis~\cite{Chenot1992,Gronostajski2019}. One of the key challenging topics mentioned in~\cite{Chenot1992} is the introduction of Model Order Reduction (MOR) methods to combat the high computational cost, which is also of paramount interest in the above-mentioned fields. Moreover, due to the multiple sources of strong non-linearities inherent in manufacturing problems,  design optimization and multi-parametric studies of large scale models turns out to be prohibitively expensive. Indeed, simulation of complex configurations can be intractable since the computational times can highly increase.

To this end, meta-model techniques are often used to tackle the computational burden. These rely on a manifold learning stage during which we need to capture the original space where the solution of the model problem lies. This data collection  consists of solving the full-scale model for an ensemble of training data over the parametric range and is commonly referred to as the offline stage. 
Even though meta-models can speed up the simulation time, nevertheless their construction with standard computations based on full-order models is expensive. 

Closely related to the concept of metamodeling, Reduced Order Models (ROMs) have been chosen to reduce the problem's dimensionality while at the same time maintaining solution accuracy. ROMs can decrease the computational complexity of large-scale systems, solving parametrized problems and offering the potential for near real-time analysis.
The methods for building ROMs can be classified into two general families: \emph{a priori} and \emph{a posteriori} ones. The well known a priori MOR  includes methods such as the Proper Generalized Decomposition (PGD)~\cite{Chinesta2011}, and the \emph{a priori} reduction method (APR)~\cite{Chinesta2011,Allery2011}. The main characteristic of all these methods is that they do not require any precomputed ROMs. 
In the second class of methods, the reduced basis is built, \emph{a posteriori}, from the state variable snapshots in the parametric space. One popular method is the POD~\cite{holmes2012turbulence,Henri2005,Aubry1991}, also known as Kharhunen-Lo\`eve Decomposition (KLD)~\cite{Karhunen1946,Loeve1977}, Singular Value Decomposition (SVD)~\cite{golub1996} or Principal Component Analysis (PCA)~\cite{Jolliffe2002,Abdi2010,Jackson1980,Jackson1981}.

For nonlinear systems, even though a Galerkin projection reduces the number of unknowns, however, the computational burden for obtaining the solution could still be high due to the computational costs involved in the evaluation of nonlinear terms. Hence, the nonlinear Galerkin projection principally leads to a ROM, but its evaluation could be more expensive than the corresponding one of the original problem. To this effect, to make the resulting ROMs computationally efficient, a sparse sampling method is used, also called hyper reduction, to mention among others, the missing point estimation (MPE)~\cite{Astrid2008}, the empirical interpolation method (EIM)~\cite{Radermacher2015}, the discrete empirical interpolation method (DEIM)~\cite{Chaturantabut2010}, the Gappy POD method~\cite{Everson1995}, and the Gauss-Newton with approximated tensors (GNAT) method~\cite{Carlberg2013}. Thus, all these methods imply the solution of a new ROM FEM problem.

In the case of a parametric analysis using POD basis interpolation on Grassmann manifolds~\cite{Amsallem2009,Mosquera2018}, the method starts with a training stage during which the problem is solved for several training points. Then, using the FEM solutions, the full-order field `snapshots' are  compressed using the POD to generate a ROM that is expected to reproduce the most characteristic dynamics of its high-fidelity counterpart solution. However, the relevant information is contained in the vector spaces generated by the (left or right) singular vectors of the snapshot matrices. Now, for a new parameter value, interpolation methods have to be defined from such relevant \emph{subspaces} spanned by the POD basis vectors~\cite{Amsallem2009}. Other approaches obviously could  be considered, such as interpolations computed on the space of matrices of a fixed rank, whereby the mechanical origin of the problem imposes to consider the vector subspaces, and not the matrices themselves~\cite{Mosquera2018}. Nevertheless, such methods as the one of interpolation between two positive semidefinite matrices of fixed rank~\cite{Bonnabel2010}, may not capture the important elements obtained from the mechanical equations.

To interpolate between different vector spaces of the same dimension (encoded into the mode $p$ of the POD), a Grassmann manifold~\cite{Lee2013} must be used, which is the set of $p$-dimensional subspaces of $\RR^n$. Such a manifold is in fact a Riemannian manifold~\cite{Gallot1990}, so we can construct \emph{geodesics} between two points, and use such geodesics to define a logarithm map to \emph{linearize}, and conversely using the exponential map to return back to the Grassmann manifold. While an interpolation cannot be done directly on Grassmann manifolds, linearization allows computing such an interpolation, at least locally once a reference point has been selected~\cite{Amsallem2009,Mosquera2018}. To any new parameter value, thus we get a new subspace obtained from  interpolation between all subspaces related to the spatial eigenvectors of the snapshot matrices.
Another approach using inverse distance weighting was initiated in~\cite{Mosquera2018,Mosquera2019b}, but it also relies on several choices (as one of the weights). Furthermore, an extension of Neville-Aitken’s algorithm to Grassmann manifolds which computes the Lagrange interpolation polynomial in a recursive way from the interpolation of two points was recently presented~\cite{Muhlbach1978}. 

In the standard POD interpolation mentioned above~\cite{Amsallem2009}, the  spatial ROM basis corresponding to the target point is used to generate a ROM FEM, which is expected to have a lower computational cost compared to the high-fidelity problem. The key idea in the \emph{Space-Time} (ST) POD basis interpolation proposed by~\cite{Lu2018,Oulghelou2021}, is that the reduced spatial and temporal basis are considered separately, both defining points on two different Grassmann manifolds. However, such points are strongly related: a spatial vector directly corresponds to a temporal vector, and \emph{vice versa}. From this, firstly we need to consider the $p$-tuples of spatial (and temporal) vectors, instead of the $p$-dimensional subspace, which defines points on an associated \emph{compact Stiefel manifold}, strongly connected to Grassmann manifolds. Contrary to what is suggested in~\cite{Oulghelou2021}, we propose a different interpolation scheme, as we do not perform interpolation of the singular values, followed by spatial and temporal calibration. Instead, we  exploit the dependence between the spatial and temporal parts. Indeed, using an interpolation algorithm defined on a Grassmann manifold, we derive \emph{curves on a compact Stiefel manifold}, which are no longer interpolating, but which nevertheless allow us to obtain new singular vectors for the spatial part, and separately for the temporal part. Such space and temporal singular vectors finally are  taken to define a \emph{mixed part} on which a classical interpolation can be computed. In the end, we get in this way a ROM matrix corresponding to a new parameter value. Note that in order to obtain a well-defined curve on compact Stiefel manifolds, we have to introduce a new procedure, the so-called \emph{oriented SVD}. Such an oriented SVD produces unique right and left eigenvectors for snapshot matrices, supposed to be \emph{generic matrices}, for which all non--zero singular values are  distinct.

The off-line stage in the ST approach consists  of solving FEM problems which are corresponding to the training points of the given parameter. The on-line stage concerns the use of a curve defined on a compact Stiefel manifold to determine the spatial and temporal ROM basis for the target point, in order to construct the related ROM snapshot matrix. In fact, the ST interpolation offers the advantage of reconstructing a snapshot matrix without relaunching ROM FEM computations. To this end, it results in near-real-time solutions due to direct matrix multiplications in the on-line stage.

We could also mention some other ST approaches~\cite{Shinde2016,Audouze2009,Choi2019,Choi2021}, where neither Grassmann nor compact Stiefel manifolds are considered. For instance, an approximation of the spatial and temporal basis functions by linear interpolation of their modes is proposed in~\cite{Shinde2016} to study the flow past a cylinder at low Reynolds numbers. A non-intrusive ROM approach for nonlinear parametrized time-dependent PDEs based on a two-level POD method by using Radial Basis Functions interpolation is presented in~\cite{Audouze2009,Audouze2013}.

The method proposed in this work is applied to a coupled thermomechanical   rigid visco-plastic (RVP) FEM analysis based on an incremental implicit approach~\cite{Kobayashi1989,Lee1973,Kobayashi1977,Feng1996}. Note that the RVP formulation specifically is  tailored for metal forming simulations, where the plastic flow is unconstrained and usually of finite magnitude, involving  large strain-rates and high temperatures. In the present study, all simulations are performed by using an in-house Matlab code which consists of two independent FEM solvers. A mechanical solver for the viscoplastic deformation analysis~\cite{Friderikos2011} and a thermal solver for the heat transfer analysis. A staggered procedure is used to solve the system of coupled equations.

The paper is organized as follows: in~\autoref{sec:POD}, the Proper Orthogonal Decomposition is presented, followed by an introduction to some basic notions about the geometry of the Grassmann and Stiefel manifolds to make the article reasonable self-contained. POD basis interpolation on Grassmannian manifolds is introduced considering the underlying  formulation of the logarithm and the exponential map. The core of this paper is illustrated in~\autoref{sec:ST_Interpolation}, where the computational framework for the ROM adaptation based on a novel non-intrusive Space-Time POD basis interpolation on compact Stiefel manifolds is developed. The following~\autoref{sec:Rigid-Viscoplastic_FEM} covers the rigid visco-plastic formulation, the general framework of the thermal field equations, and the thermomechanical coupling. In~\autoref{sec:Numerical_Investigations}, the interpolation performance applied to a metal forming process is shown, as well as further computational aspects are discussed. Finally,~\autoref{sec:Conclusions} highlights the main results and some important outcomes. 

\section{Space--Time POD, Grassmann and compact Stiefel manifolds}\label{sec:POD}

Let us recall here the important link between Proper Orthogonal Decomposition and Grassmann manifold~\cite{Amsallem2009,Edelman1998g,Absil2004,Mosquera2018,Mosquera2019b}. 

Assume $\bS\in \VecMat{n}{m}$ to be any real matrix of size $n\times m$ (with $n\geq m$), taken here to be a snapshot matrix with $n=3N_S$ obtained from the spatial discretization $N_s$, and $m=N_t$ obtained from the time one. Any spatial POD of mode $p$ leads to a $p$-dimensional vector space $\mathcal{V}_p\subset \RR^{m}$ such that the Frobenius norm
\begin{linenomath}
\begin{equation*}
\|\bS-\boldsymbol{\Pi}_{p}\bS\|_{\text{F}}^{2}
\end{equation*} 
\end{linenomath}
is minimal, where matrix $\boldsymbol{\Pi}_{p}$ corresponds to the orthogonal projection on $\mathcal{V}_p$ (see \cite{Mosquera2018} for more details). Such a matrix $\boldsymbol{\Pi}_{p}$ is directly obtained from a Singular Value Decomposition (SVD) of $\bS$. Indeed, let us write a SVD
\begin{linenomath} \begin{equation*}
	\bS=\bPhi \bSig \bPsi^T
\end{equation*} \end{linenomath}
with $\bPhi=[\phi_1,\dots,\phi_r]$ and $\bPsi=[\psi_1,\dots,\psi_r]$, where the columns $\phi_k\in \RR^n$ and $\psi_k\in \RR^m$ form a set of orthonormal vectors, and $\bSig \in \VecMat{r}{r}$ is a diagonal matrix, where $r$ denotes the rank of $\bS$. Then, we can define $\bPhi_{p}:=[\phi_1,\dots,\phi_p]\in\VecMat{n}{p}$ and we obtain $\boldsymbol{\Pi}_{p}=\bPhi_{p}\bPhi_{p}^{T}$.

In this classical approach, the relevant object is not the reduced matrix $\bS_p:=\boldsymbol{\Pi}_{p}\bS$, supposed to be of maximal rank, but the $p$-dimensional vector space $\mathcal{V}_p$ spanned by vectors $\phi_1,\dots,\phi_p$, and thus the image of the matrix $\bPhi_{p}$. From this, interpolation has to be considered on the set of all $p$-dimensional vector spaces, that is on the so--called \emph{Grassmann manifold} $\mathcal{G}(p,n)$:
\begin{linenomath} \begin{equation*}
	\mathcal{G}(p,n):=\left\{\mathcal{V}_p\subset \RR^n,\quad \dim(\mathcal{V}_p)=p\right\}.
\end{equation*} \end{linenomath}
	
Note here that the \emph{point} $\bom:=\mathcal{V}_p\in \mathcal{G}(p,n)$   defines a vector space spanned by the set $\phi_1,\dots,\phi_p$  \emph{represented} by matrix $\bPhi_{p}$, however this matrix representation is not unique (see Example~\ref{ex:Point_on_Grassman}).
 
Take now a set $\{\lambda_{1},\dotsc,\lambda_{N}\}$ of parameter values leading to snapshot matrices $\bS^{(1)},\dotsc,\bS^{(N)}$ with SVD
\begin{linenomath} \begin{equation*}
		\bS^{(k)}=\bPhi^{(k)} \bSig^{(k)} \bPsi^{(k)},\quad \bPhi^{(k)}=[\phi_1^{(k)},\dots,\phi_{r}^{(k)}],\quad \bPsi^{(k)}=[\psi_1^{(k)},\dots,\psi_r^{(k)}],
\end{equation*} \end{linenomath}
where $\phi_i^{(k)}$ are orthonormal vectors in $\RR^n$ and $\psi_j^{(k)}$ are  orthonormal vectors in $\RR^m$.

The classical approach~\cite{Amsallem2009,Mosquera2018} then considers the spatial POD of the snapshot matrices $\bS_p^{(1)},\dotsc,\bS_p^{(N)}$ of mode $p$, so that we obtain points $\bom_{i}$ ($i=1,\dotsc,N$) on $\mathcal{G}(p,n)$, respectively represented by the matrices 
\begin{linenomath} \begin{equation*}
	\bPhi_{p}^{(k)}:=[\phi_1^{(k)},\dots,\phi_p^{(k)}]\in \VecMat{n}{p},\quad \left(\bPhi_{p}^{(k)}\right)^T \bPhi_{p}^{(k)}=\mathbf{I}_p.
\end{equation*} \end{linenomath}

To any new parameter value $\wtlambda$, it is possible to make an interpolation considering the spatial part based on the points $\bom_i \in \mathcal{G}(p,n)$, using a local chart given by normal coordinates~\cite{Amsallem2009,Mosquera2018,Mosquera2019b}, in order to obtain a point $\widetilde{\bom}\in \mathcal{G}(p,n)$ represented by a matrix $\widetilde{\bPhi}$. From such a point $\widetilde{\bom}\in\mathcal{G}(p,n)$, we deduce a $p$-dimensional vector space on which some POD-Galerkin approach~\cite{Mosquera2018} can lead to a new ROM model.

On the contrary, we propose  another approach as we consider a \emph{Space--Time interpolation}, using both the spatial vector spaces represented by matrices $\bPhi_{p}^{(k)}$ and the temporal vector spaces represented by matrices
\begin{linenomath} \begin{equation*}
	\bPsi_{p}^{(k)}:=[\psi_1^{(k)},\dots,\psi_p^{(k)}]\in \VecMat{m}{p},\quad \left(\bPsi_{p}^{(k)}\right)^T \bPsi_{p}^{(k)}=\mathbf{I}_p.
\end{equation*} \end{linenomath} 
An important observation now is that matrices $\bPhi_{p}^{(k)}$ (resp. $\bPsi_{p}^{(k)}$) directly define an \emph{ordered $p$-tuple of orthonormal vectors} in $\RR^n$ (resp. $\RR^m$), that is a point on the \emph{compact Stiefel manifold}
\begin{linenomath} \begin{equation*}
	\StieComp{p}{n}:=\left\{\text{Ordered orthonormal } p \text{-tuple of vectors in } \RR^n\right\}.
\end{equation*} \end{linenomath}
To obtain a Space-Time POD interpolation (instead of a spatial POD interpolation followed by Galerkin approach), we finally adopted the following strategy, when dealing with a parameter value $\widetilde{\lambda}$:
\begin{enumerate}
	\item Define a \emph{curve} on the compact Stiefel manifold corresponding to the spatial part
		\begin{linenomath} \begin{equation*}
			\lambda\mapsto \bPhi(\lambda)\in  \StieComp{p}{n}
		\end{equation*} \end{linenomath} 
		obtained using the already known interpolation algorithm on Grassmann manifold.
	\item In the same way, define a \emph{curve} on the compact Stiefel manifold corresponding to the temporal part
	\begin{linenomath} \begin{equation*}
		\lambda\mapsto \bPsi(\lambda)\in  \StieComp{p}{m}.
	\end{equation*} \end{linenomath} 
	\item Construct an \emph{interpolated curve} $\lambda\mapsto \bS(\lambda)$ passing through the POD of mode $p$ snapshot matrices $\bS_p^{(k)}$, in order to obtain an interpolation of a ROM matrix $\widetilde{\bS}:=\bS(\widetilde{\lambda})$.
\end{enumerate}
In the next subsections, we give all important details to obtain such an interpolated curve $\lambda\mapsto \bS(\lambda)$. First, in~\autoref{subsec:Rieman_Geo_Grass} we explain how to compute on Grassmann manifolds using their Riemannian structure to obtain explicit formulae for the \emph{geodesics} defining \emph{normal coordinates}. From this explicit formulae, we can deduce in~\autoref{subsec:Target_Algo_Stiefel} a \emph{target algorithm} in order to define the curves 
\begin{linenomath} \begin{equation*}
	\lambda\mapsto \bPhi(\lambda)\in  \StieComp{p}{n},\quad \lambda\mapsto \bPsi(\lambda)\in  \StieComp{p}{m}
\end{equation*} \end{linenomath}
on compact Stiefel manifolds. The question on how to define an interpolated curve for matrices $\bS_p^{(k)}$ will then be addressed in \autoref{sec:ST_Interpolation}.

%% --------------------------------------------
\subsection{Riemannian geometry on Grassmann manifolds}%
\label{subsec:Rieman_Geo_Grass}
%% --------------------------------------------

We will summarize now some essential results about Grassmann manifolds. Such manifolds are in fact \emph{complete Riemannian manifolds}~\cite{Gallot1990}, meaning for instance that we can define the \emph{length} of a curve. Moreover, we can always construct a curve of the shortest length between two points, which is called a \emph{geodesic}, and it will be the starting point to define \emph{normal coordinates} via the exponential and logarithm map (Definition~\ref{def:Exp_Map} and~\ref{def:Log_map_Grass}). As we cannot do direct computations on Riemann manifolds, normal coordinates enable us to obtain formulae of curves, such as the Lagrangian polynomials.
Note finally that a rigorous mathematical background of all of this is given in~\cite{Kozlov1997}. 

After we give a definition of the Grassmann manifold and how to \emph{represent} its points with matrices, we propose to define the \emph{tangent plane} using matrix representative, to have formulae for a scalar product, given by~\eqref{eq:rieman_scalar_product_grass}. From this, we deduce a classical expression for geodesics (Theorem~\ref{thm:Eq_Geo_Grass}).

Let $p\leq n$ be two non-zero integers and $\mathcal{G}(p,n)$ the Grassmann manifold of $p$-dimensional subspaces in $\RR^{n}$. In fact, Grassmann manifolds are special cases of \emph{quotient manifolds}, meaning that a point on such a manifold can have many \emph{representatives}. Let us consider indeed a $p$-dimensional linear subspace $\mathcal{V}$ of $\RR^n$. Such a subspace can be defined using any ordered set of $p$ independent vectors $\vv_1,\dots,\vv_p$ in $\RR^n$, encoded into a full rank matrix
\begin{linenomath} \begin{equation*}
	\mathbf{M}:=[\vv_1,\dots,\vv_p]\in \VecMat{n}{p}.	
\end{equation*} \end{linenomath}
Any other basis $\vv'_1,\dots,\vv'_p$ of $\mathcal{V}$ will then lead to another full rank matrix
\begin{linenomath} \begin{equation*}
	\mathbf{M}':=[\vv'_1,\dots,\vv'_p]\in \VecMat{n}{p},	
\end{equation*} \end{linenomath} 
and we necessary have
\begin{linenomath} \begin{equation*}
	\mathbf{M}'=\mathbf{M}\mathrm{P}
\end{equation*} \end{linenomath}
where $\mathrm{P}\in \mathrm{GL}(p)$ is some invertible matrix in $\VecMat{p}{p}$. From all this, we deduce that the point $\bom:=\mathcal{V}\in \mathcal{G}(p,n)$ is represented by the infinite set of matrices
\begin{linenomath} \begin{equation*}
	\left\{ \bM\mathrm{P},\quad \mathrm{P}\in \mathrm{GL}(p)\right\}.
\end{equation*} \end{linenomath}

Now, the ordered set of $p$ independent vectors in $\RR^n$ and thus the set of full rank matrices in $\VecMat{n}{p}$ define the Stiefel manifold (see Figure~\ref{fig:Two_Points_Stiefel})	
\begin{linenomath} \begin{equation*}
	\mathcal{S}t(p,n):=\left\{\mathbf{M}=[\vv_1,\dots,\vv_p]\in \VecMat{n}{p},\quad \text{rg}(\mathbf{M})=p\right\}
\end{equation*} \end{linenomath} 
so that we obtain a natural map from such Stiefel manifold and the Grassmann manifold $\mathcal{G}(p,n)$ (see Figure~\ref{fig:Two_Points_Grassmann}):
\begin{linenomath} \begin{equation*}
	\mathbf{M}=[\vv_1,\dots,\vv_p]\in \mathcal{S}t(p,n)\mapsto \bom=\left\{ \bM\mathrm{P},\quad \mathrm{P}\in \mathrm{GL}(p)\right\}.
\end{equation*} \end{linenomath}
In our situation, nevertheless, we will only focus on \emph{orthonormal bases} of $p$-dimensional subspaces. Doing so, we thus consider matrices defined by orthonormal vectors, leading to the so-called \emph{compact Stiefel manifold} 
\begin{equation}\label{eq:Stief_Comp}
\StieComp{p}{n}:=\left\{ \bY\in \VecMat{n}{p},\quad \bY^{T}\bY=\mathbf{I}_{p}\right\}
\end{equation}
and any point $\bom\in \mathcal{G}(p,n)$ will then be represented by the infinite set
\begin{linenomath} \begin{equation*}
	\left\{ \bY \mathrm{Q},\quad \mathrm{Q}\in \mathrm{O}(p)\right\}
\end{equation*} \end{linenomath}
where $\bY=[\by_1,\dots,\by_p]$ is defined using an orthonormal basis $\by_1,\dots,\by_p$ of $\bom$. This defines a \emph{surjective} map
\begin{linenomath} \begin{equation*}
\pi \: : \: \bY\in \StieComp{p}{n}\mapsto \bom= \pi(\bY)=\left\{ \bY \mathrm{Q},\quad \mathrm{Q}\in \mathrm{O}(p)\right\}\in \mathcal{G}(p,n)
\end{equation*} \end{linenomath}
and the set of all matrices representing the same point $\bom\in \mathcal{G}(p,n)$ is called the \emph{fiber} of $\pi$ at $\bom$ (see Figure~\ref{fig:Grasmann_Quotient} for an illustration of a fiber): 
\begin{linenomath} \begin{equation*}
\pi^{-1}(\bom)=\left\{ \bY Q,\quad Q\in \mathrm{O}(p)\right\}.
\end{equation*} \end{linenomath}

\begin{remark}
An important point here is that, from now on, any computation on $\mathcal{G}(p,n)$ will be done using a \emph{choice in the fibers}. Nevertheless, for any point $\bom\in \mathcal{G}(p,n)$, \emph{there is no canonical} way to choose an element $\bY\in \pi^{-1}(\bom)$, so any computation has to be independent of that choice.  
\end{remark}

We need now to define the \emph{geodesics} of Grassmann manifold, which can be done once we have defined the \emph{tangent plane} at each point $\bom\in \mathcal{G}(p,n)$ and a Riemaniann metric. Take any point $\bom \in \mathcal{G}(p,n)$ represented by a matrix $\bY=[\by_1,\dots,\by_p]$ of orthonormal vectors, the tangent plane $\mathcal{T}_{\bom}:=\mathcal{T}_{\bom}\mathcal{G}(p,n)$ is then represented by the $p(n-p)$ dimensional vector space
\begin{equation}\label{eq:Hor_Space}
	\text{Hor}_{\bY}:=\left\{ \bZ\in \VecMat{n}{p},\quad \bZ^{T}\bY=0\right\},
\end{equation}
called the \emph{horizontal space}, where $\bY\in \pi^{-1}(\bom)$. From all this, a vector $v\in \mathcal{T}_{\bom}$ will be called a \emph{velocity vector}, which can be \emph{represented} by a matrix $\bZ\in \VecMat{n}{p}$ such that $\bZ^{T}\bY=0$, and $\bZ$ is called a \emph{horizontal lift} of $v$.

\begin{example}\label{ex:Point_on_Grassman}
	Take here $p=2$ and $n=5$, so that $\mathcal{G}(2,5)$ is the set of planes in a five dimensional space. The matrices
	\begin{linenomath} \begin{equation*}
		\bY=\begin{bmatrix}
		\frac{1}{2} & 0 \\
		-\frac{1}{2} & \frac{\sqrt{2}}{2} \\
		0 & 0 \\
		\frac{1}{2} & \frac{\sqrt{2}}{2} \\
		\frac{1}{2} & 0 \\
		\end{bmatrix},
		\quad
		\bY'=\left[ \begin {array}{cc} 
		\frac{\sqrt {2}}{4}&-\frac{\sqrt {2}}{4} \\ 
		\frac{2-\sqrt {2}}{4}&\frac{2+\sqrt {2}}{4} \\ 
		0&0 \\ 
		\frac{2+\sqrt {2}}{4} &\frac{2-\sqrt {2}}{4} \\ 
		\frac{\sqrt {2}}{4}&-\frac{\sqrt {2}}{4}
		\end {array} \right] 
	\end{equation*} \end{linenomath}
are in the compact Stiefel manifold $\StieComp{2}{5}$, representing the same plane $\bom\in \mathcal{G}(2,5)$. The horizontal space $\text{Hor}_{\bY}$ defined by~\eqref{eq:Hor_Space} is a 6-dimensional vector space of matrices $\bZ$, for instance given by
\begin{linenomath} \begin{equation*}
	\bZ=\small\begin{bmatrix}
	u_1 & v_1 \\
	u_2 & v_2 \\
	u_3 & v_3 \\
	-u_2 & -v_2 \\ 
	-u_1+u_2-u_4\hspace*{0.5cm} & -v_1+v_2-v_4
	\end{bmatrix},\quad u_i,v_i\in \RR.
\end{equation*} \end{linenomath} 	
\end{example}

Taking now velocity vectors $v_1,v_2\in \mathcal{T}_{\bom}$ with respective horizontal lifts $\bZ_1,\bZ_2\in \text{Hor}_{\bY}$ we define the point--wise scalar product~\cite{Wong1967,Absil2004}:
\begin{equation}\label{eq:rieman_scalar_product_grass}
\langle v_1,v_2\rangle_{\bom}:=\tr\left(\bZ_1^{T}\bZ_2\right).
\end{equation}

Such a Riemannian metric leads to explicit geodesics given by~\cite{Absil2004,Edelman1998g}:
\begin{theorem}\label{thm:Eq_Geo_Grass}
Let $\bom\in \mathcal{G}(p,n)$ represented by $\bY\in \StieComp{p}{n}$.  For any $v\in \mathcal{T}_{\bom}$ with horizontal lift given by $\bZ$ in $\text{\emph{Hor}}_{\bY}$, let $\bZ=\bU\bSig\bV^{T}$ be a thin SVD of $\bZ$. Then 
\begin{equation}\label{eq:geo_v}
\alpha_{v} : \: t\in \RR\mapsto \alpha_{v}(t):=\pi\left[\left(\bY\bV\cos(t\bSig)+\bU\sin(t\bSig)\right)\bV^{T}\right]  \in \mathcal{G}(p,n) \\
\end{equation}
is the unique maximal geodesic such that $\alpha_{v}(0)=\bom$ and initial velocity
\begin{linenomath} 
\begin{equation*}
 	\dot{\alpha_{v}}(0):=\frac{\partial \alpha_{v}(t)}{\partial t}\mid_{t=0}=v. 
\end{equation*}
\end{linenomath} 
\end{theorem}

\begin{remark}
Up to our knowledge, there is no proof that 
\begin{equation}\label{eq:Geo_path}
\bY(t):=\left(\bY\bV\cos(t\bSig)+\bU\sin(t\bSig)\right)\bV^{T}\in \StieComp{p}{n}.
\end{equation} 
In fact, this follows by direct computation. Indeed, $\bZ=\bU\bSig\bV^{T}$ being a thin SVD, we have $\bV\in \mathrm{O}(p)$ and
\begin{linenomath} \begin{equation*}
\bZ^{T}\bY=\bV\bSig\bU^{T}\bY=0\implies \bSig\bU^{T}\bY=0
\end{equation*} \end{linenomath}
so that 
\begin{linenomath} \begin{equation*}
\sin(t\bSig)\bU^{T}\bY=0 \text{ and } \bY^{T}\bU\sin(t\bSig)=0.
\end{equation*} \end{linenomath}
Finally, we have: 
\begin{multline*}
\bY^{T}(t)\bY(t)=
\bV\left( \cos^{2}(t\bSig)+\underbrace{\sin(t\bSig)\bU^{T}\bY}_{=0}\bV\cos(t\bSig)+\right. \\
\left.\cos(t\bSig)\underbrace{\bV^{T}\bY^{T}\bU\sin(t\bSig)}_{=0}+\sin^2(t\bSig)\right)\bV^{T}
\end{multline*}
which concludes the proof.
\end{remark}

\begin{remark}\label{rem:Exp_in_Stief}
In many cases, formulas of the geodesic do not use the right multiplication by $\bV^{T}$, as for instance in~\cite{Absil2004,Mosquera2018}. Of course, as $\bV$ being in $\mathrm{O}(p)$ both matrices \begin{linenomath} \begin{equation*}
\left(\bY\bV\cos(t\bSig)+\bU\sin(t\bSig)\right)\bV^{T} \text{ and } \bY\bV\cos(t\bSig)+\bU\sin(t\bSig)
\end{equation*} \end{linenomath}
define the same point on $\mathcal{G}(p,n)$. Now, the choice of such right multiplication in~\eqref{eq:Geo_path} is related to the choice of the horizontal lift $\bZ=\bU\bSig\bV^{T}$. Indeed, taking back the path given by~\eqref{eq:Geo_path}, we have
\begin{linenomath} \begin{equation*}
\dot{\bY}(t)=\left(-\bY\bV\bSig\sin(t\bSig)+\bU\bSig\cos(t\bSig)\right)\bV^{T}\implies \dot{\bY}(0)=\bU\bSig\bV^{T}=\bZ
\end{equation*} \end{linenomath}
which corresponds to the choice of the horizontal lift for velocity vector $v\in \mathcal{T}_{\bom}$.
\end{remark}
A consequence of Theorem~\ref{thm:Eq_Geo_Grass} is an explicit formula for the exponential map~\cite{Absil2004,Mosquera2018} (see Figure~\ref{fig:Log_Exp_map}):
\begin{definition}\label{def:Exp_Map}
Let $\bom\in \mathcal{G}(p,n)$ be represented by $\bY\in \StieComp{p}{n}$. For any velocity vector $v\in \mathcal{T}_{\bom}$ with horizontal lift $\bZ\in \text{\emph{Hor}}_{\bY}$, take $\bZ=\bU\bSig\bV^{T}$ to be a thin SVD of $\bZ$. Then we define the exponential map
\begin{multline*}
\Exp_{\bom} \: : \: \mathcal{T}_{\bom}\longrightarrow \mathcal{G}(p,n),\\ 
v\mapsto \Exp_{\bom}(v):=\pi\left[\left(\bY\bV\cos(\bSig)+\bU\sin(\bSig)\right)\bV^{T}\right]=\alpha_{v}(1).
\end{multline*}
\end{definition}

Now, it is possible to define directly some inverse map of the exponential map, called the logarithm map~\cite{Absil2004}, but only \emph{locally}. For any $\bom$ and $\bY$ in its fiber, let us first define the open space
\begin{equation}\label{eq:Open_set_Um}
\mathrm{U}_{\bom}:=\{ \bom_{1}\in \mathcal{G}(p,n),\quad \bY^{T}\bY_{1} \text{ is invertible},\quad \bY_{1}\in \pi^{-1}(\bom_{1})\}.
\end{equation}
Then we have:

\begin{definition}[Logarithm map on Grassmannian manifold]\label{def:Log_map_Grass}
Let $\bom\in \mathcal{G}(p,n)$ be represented by a matrix $\bY\in  \StieComp{p}{n}$. For any point $\bom_1$ in the open space $\mathrm{U}_{\bom}$ represented by a matrix $\bY_1\in \StieComp{p}{n}$, define a thin SVD
\begin{linenomath} 
	\begin{equation*}
	\bY_{1}\left(\bY^{T}\bY_{1}\right)^{-1}-\bY=\bU\bSig \bV^{T}.
	\end{equation*} 
\end{linenomath}
Then the logarithm $\Log_{\bom}(\bom_{1})\in \mathcal{T}_{\bom}$ is the velocity vector in $\mathcal{T}_{\bom}$ with horizontal lift
\begin{linenomath} 
	\begin{equation*}
	\bZ=\bU \arctan(\bSig) \bV^{T}\in \text{\emph{Hor}}_{\bY}.
	\end{equation*}
\end{linenomath} 
\end{definition}

\begin{remark}\label{rem:Log_and_Exp_Inv}
The logarithm map is only defined on some \emph{open set} $\mathrm{U}_{\bom}$. This means that for any point $\bom_1\notin \mathrm{U}_{\bom}$, the associated matrix $\bY^{T}\bY_{1}$ is not invertible, so that the computation of
\begin{linenomath} 
	\begin{equation*}
	\bY_{1}\left(\bY^{T}\bY_{1}\right)^{-1}-\bY
	\end{equation*}
\end{linenomath} 
can not be done. Note finally that such an open set is strongly related to the \emph{cut-locus} of a Grassmann manifold~\cite{Berceanu1997}. 
\end{remark}

%% -----------------------------------------
\subsection{Target Algorithm on compact Stiefel manifolds}%
\label{subsec:Target_Algo_Stiefel}
%% -----------------------------------------
All the mathematical background  summarized in~\autoref{subsec:Rieman_Geo_Grass} can be used to obtain an interpolation curve between points $\bom_1,\dotsc,\bom_N$ on Grassmann manifold $\mathcal{G}(p,n)$~\cite{Amsallem2009,Mosquera2018}, where each point $\bom_i$ corresponds to a parameter value  $\lambda_i$. Indeed, once a \emph{reference point} $\bom_{i_0}\in \left\{\bom_1,\dotsc,\bom_N \right\}$ is chosen~(see Figure~\ref{fig:Log_Exp_map}): 
\begin{itemize}
	\item[$\bullet$] We use the logarithm map $\Log_{\bom_{i_0}}$ to \emph{linearize}, i.e., meaning we define velocity vectors $v_i:=\Log_{\bom_{i_0}}(\bom_i)$ on the \emph{vector space} $\mathcal{T}_{\bom_{i_0}}$.
	\item[$\bullet$] We obtain an interpolation curve $\lambda \mapsto v(\lambda)$ between vectors $v_i$, using for instance Lagrangian polynomial, and thus
	\begin{linenomath} \begin{equation*}
		v(\lambda_{i})=v_i,\quad \forall i=1,\dots,N.
	\end{equation*} \end{linenomath}
	\item[$\bullet$] Taking the exponential map $\Exp_{\bom_{i_0}}$, we obtain back an interpolation curve 
	\begin{linenomath} \begin{equation*}
		\lambda\mapsto \bom(\lambda):=\Exp_{\bom_{i_0}}(v(\lambda))	
	\end{equation*} \end{linenomath}
	between the points $\bom_1,\dotsc,\bom_N$ on $\mathcal{G}(p,n)$, so that
	\begin{linenomath} \begin{equation*}
		\bom(\lambda_{i})=\bom_i,\quad \forall i=1,\dots,N.
	\end{equation*} \end{linenomath}
\end{itemize}

We propose here to define curves on the compact Stiefel manifold $\mathcal{S}t(p,n)$ instead of the ones defined on the Grassmann manifold $\mathcal{G}(p,n)$. The starting point is a set of matrices $\bY_{1},\dotsc,\bY_{N}$ in the compact Stiefel manifold $\mathcal{S}t(p,n)$, corresponding to parameter values $\lambda_{1},\dots,\lambda_{N}$. Once a reference parameter value $\lambda_{\i_0}$ has been chosen, we obtain a curve
\begin{linenomath} \begin{equation*}
	\lambda \mapsto \bY(\lambda)
\end{equation*} \end{linenomath} 
where in general,
\begin{linenomath} \begin{equation*}
	\bY(\lambda_{i})\neq \bY_i. 
\end{equation*} \end{linenomath}
As a consequence, such a curve will not be an interpolation curve between the matrices $\bY_{1},\dotsc,\bY_{N}$ (see Remark~\ref{rem:Target_Not_Interpolated_Curve}).  Before doing so, and to obtain well-defined curves, we need to make a specific definition:

\begin{definition}[Genericity]\label{def:Genericity}
A matrix is said to be \emph{generic} if all its non--zero singular values are distinct. The set of generic matrices in $\VecMat{n}{p}$ is denoted $\VecMatG{n}{p}$.
\end{definition}

For any generic matrix $\bM\in \VecMatG{n}{p}$, we know that its thin SVD $\bM=\bU\bSig\bV^{T}$ is well defined. Indeed, taking $\sigma_1>\dotsc>\sigma_p$ to be its ordered singular values, we can write 
\begin{equation}\label{eq:SVD_Generic}
	\bM=\sum_{i=1}^p \sigma_{i} \uu^i \vv_i^{T}
\end{equation}
where $\uu_i$ (resp. $\vv_i$) is a left singular vector associated to $\sigma_{i}$ (resp. a right singular vector). All singular values being distinct, the only other possibility is to consider singular vectors $\epsilon_i\uu_i$ and $\epsilon_i\vv_i$, with $\epsilon_i=\pm 1$, so that the decomposition~\eqref{eq:SVD_Generic} remains the same. We thus deduce that the target Algorithm below is well defined:

\begin{algo}[Target algorithm]\label{alg:Target_Alg}

\begin{itemize}
	\item[$\bullet$] \textbf{Inputs}:
		\begin{itemize}
			\item Matrices $\bY_{1},\dotsc,\bY_{N}$ in $\StieComp{p}{n}$, corresponding to parameter values $\lambda_{1}<\dotsc <\lambda_{N}$.
			\item A reference parameter value $\lambda_{i_0}$ with $i_0\in \{ 1,\dotsc,N\}$. 
			\item A parameter value $\lambda$.
		\end{itemize}
	\item[$\bullet$] \textbf{Output}: A matrix $\bY(\lambda)\in \StieComp{p}{n}$. 
\end{itemize}

\begin{enumerate}
\item Define $\bZ_{i_0}:=\mathbf{0}$ and for each $k\in \{1,\dotsc,N\}$ with $k\neq i_0$ compute a thin SVD of the generic matrix
\begin{linenomath} \begin{equation*}
	\bY_{k}\left(\bY_{i_0}^{T}\bY_{k}\right)^{-1}-\bY_{i_0}=\bU_k \bSig_k  \bV_k^{T}
\end{equation*} \end{linenomath}
and define
\begin{linenomath} \begin{equation*}
\bZ_{k}:=\bU_{k} \arctan(\bSig_{k}) \bV_{k}^{T},\quad \text{ with assumption } \bZ_{k}\in \VecMatG{n}{p}.
\end{equation*} \end{linenomath}

\item Define an $n\times p$ matrix and compute a thin SVD
\begin{linenomath} \begin{equation*}
\bZ(\lambda):=\sum_{i=1}^{N} \prod_{i \neq j} \frac{\lambda-\lambda_j}{\lambda_i-\lambda_j}\mathbf{Z}_i=\bU(\lambda)\bSig(\lambda)\bV(\lambda)^T, 
\end{equation*} \end{linenomath}
with assumption $\bZ(\lambda)\in \VecMatG{n}{p}$.
\item Define the $n\times p$ matrix in $\StieComp{p}{n}$ (see Remark~\ref{rem:Exp_in_Stief}):
\begin{equation}\label{eq:Y_lambda}
\mathbf{Y}(\lambda) := 
[\mathbf{Y}_{i_0} \bV(\lambda)\text{cos} (\bSig(\lambda)) + 
\bU(\lambda) \text{sin} (\bSig(\lambda))]  \bV(\lambda)^T.
\end{equation}
Note: cos and sin act only on diagonal entries.
\end{enumerate}
\end{algo}

In this algorithm, as already noticed and following the assumptions of genericity, the matrices $\bZ_{k}$, $\bZ(\lambda)$ and $\bY(\lambda)$ do not depend on the choice of matrices in the associated thin SVD. 

\begin{remark}\label{rem:Target_Not_Interpolated_Curve}
	Using this target Algorithm to parameter value $\lambda:=\lambda_{k}$ leads to some matrix $\bY(\lambda_{k})$ generally different from $\bY_{k}$ (except for $k=i_0$). Thus, such an algorithm computed on compact Stiefel manifold do not produce an \emph{interpolation} on the points $\bY_{1},\dotsc,\bY_{N}$ (see Figure~\ref{fig:Grasmann_Quotient}). Indeed, to represent an interpolation curve between these points means that if we consider the parameter value $\lambda=\lambda_k$ (with $k\in \{1,\dots,N\}$) as input in the algorithm, one should expect to return as output  $\bY(\lambda_{k})$ (given by~\eqref{eq:Y_lambda}) the initial matrix $\bY_k$, which is not the case in general.  
	
	Nevertheless, matrices $\bY(\lambda_{k})$ and $\bY_k$ define the same point on the Grassmann manifold $\mathcal{G}(p,n)$, meaning that they both define an \emph{orthonormal basis of the same subspace} $\bom_k$ (see Remark~\ref{rem:Log_and_Exp_Inv}). As a consequence, a projection matrix onto the subspace $\bom_k$ is given by $\bY(\lambda_k)^{T}\bY(\lambda_k)$ or equivalently by $\bY_k^{T}\bY_k$. 
\end{remark}

\begin{example}
	Take for instance the compact Stiefel manifold $\StieComp{2}{5}$, and the three matrices 
	\begin{linenomath} \begin{equation*}
		\bY_1:=\begin{bmatrix}
		1 & 0  \\
		0 & 1  \\
		0 & 0  \\
		0 & 0  \\
		0 & 0 
		\end{bmatrix},\quad 
		\bY_2:=\begin{bmatrix}
		\frac{\sqrt{3}}{3} &  \frac{\sqrt{3}}{3} \\
		0 & \frac{\sqrt{3}}{3}  \\
		\frac{\sqrt{3}}{3} & 0 \\
		-\frac{\sqrt{3}}{3} & \frac{\sqrt{3}}{3} \\
		0 & 0 
		\end{bmatrix},\quad
		\bY_3:=\begin{bmatrix}
		\frac{\sqrt{3}}{3} &  -\frac{\sqrt{6}}{6} \\
		0 & \frac{\sqrt{6}}{4}  \\
		\frac{\sqrt{3}}{3} & \frac{\sqrt{6}}{12} \\
		0 & \frac{\sqrt{6}}{4} \\
		\frac{\sqrt{3}}{3} & \frac{\sqrt{6}}{12} 
		\end{bmatrix}
	\end{equation*} \end{linenomath}
	which correspond respectively to $\lambda_{1}=15$, $\lambda_{2}=22$ and $\lambda_{3}=27$. Choosing the reference parameter value to be $\lambda_1$ and following the target Algorithm~\ref{alg:Target_Alg} we obtain
	\begin{linenomath} \begin{equation*}
		\bY_2(\bY_1^T\bY_2)^{-1}-\bY_1=\begin{bmatrix}
		0 & 0  \\
		0 & 0  \\
		1 & -1  \\
		-1 & 2  \\
		0 & 0
		\end{bmatrix},\quad 
		\bY_3(\bY_1^T\bY_3)^{-1}-\bY_1=\begin{bmatrix}
		0 & 0  \\
		0 & 0  \\
		1 & 1  \\
		0 & 1  \\
		1 & 1
		\end{bmatrix}
	\end{equation*} \end{linenomath}
	Taking $\lambda=\lambda_2$ and $\lambda=\lambda_3$ as inputs in the algorithm, we finally obtain the matrices (with computation done using 5 digits):
	\begin{linenomath} \begin{equation*}
	\bY(\lambda_{2})=
	\begin{bmatrix}
	0.77460 & 0.25820 \\
	0.25820 & 0.51640 \\ 
	0.51640 & -0.25820 \\ 
	- 0.25820 & 0.77460 \\ 
	0 & 0 
	\end{bmatrix}\neq \bY_2,\quad
	\bY(\lambda_{3})=
	\begin{bmatrix}
	0.67860 & -0.19876 \\ 
	-0.19876 & 0.57922 \\ 
	0.47984 & 0.38046 \\ 
	-0.19876 & 0.57922 \\ 
	0.47984 & 0.38046
	\end{bmatrix}\neq \bY_3.
	\end{equation*} \end{linenomath}
\end{example}

%% ------------------------------------------------------
\section{Space-Time Interpolation on compact Stiefel manifolds}
\label{sec:ST_Interpolation}
%% ------------------------------------------------------
As already noticed, POD is extracting the optimal space structures and the associated temporal modes. An important property is that the spatial and temporal orthogonal modes are \emph{coupled}: each space component is associated with a temporal component partner and there is a one-to-one correspondence between both spaces.
Taking advance of this decomposition into orthogonal modes, it is natural to try a \emph{Space-Time} interpolation on compact Stiefel manifolds based on the target Algorithm~\ref{alg:Target_Alg}, instead of an interpolation of the space part alone, followed by a Galerkin approach as is classically done~\cite{Amsallem2009,Mosquera2018}. 

As a starting point, take a set of \emph{snapshot matrices} $\bS^{(1)},\dotsc,\bS^{(N)}$, where each matrix $\bS^{(k)}\in \VecMat{n}{m}$ corresponds to a given parameter value $\lambda_{k}\in \RR$, with $\lambda_1<\dotsc<\lambda_{N}$ and $n=3N_s$ corresponding to the spatial part, while $m=N_t$ corresponds to the temporal part. For a given mode $p\leq N_t$, our goal is to
\begin{enumerate}
	\item Extract in a unique way a POD of mode $p$ of each matrix $\bS^{(k)}$, so that we have a well defined map
	\begin{linenomath} \begin{equation*}
		\bS^{(k)}\in \VecMat{n}{m} \mapsto \bS^{(k)}_p \in \VecMat{n}{m}.
	\end{equation*} \end{linenomath}
	\item Obtain for each $\bS^{(k)}_p \in \VecMat{n}{m}$ a unique matrix $\bPhi_{p}^{(k)}\in \StieComp{p}{n}$ for the spatial part and another unique matrix $\bPsi_{p}^{(k)}\in \StieComp{p}{m}$ for the temporal part. 
	\item Use the target Algorithm~\ref{alg:Target_Alg} on matrices $\bPhi_{p}^{(k)}$ first, and then on matrices $\bPsi_{p}^{(k)}$, in order to obtain two curves 
	\begin{equation}\label{eq:Spatial_Time_Curve}
	 \lambda\mapsto \bPhi(\lambda),\quad \lambda\mapsto\bPsi(\lambda)
	\end{equation}
	\emph{which are not interpolated curves}, as in general $\bPhi(\lambda_k)\neq \bPhi_{p}^{(k)}$ and $\bPsi(\lambda_k)\neq \bPsi_{p}^{(k)}$ (see Remark~\ref{rem:Target_Not_Interpolated_Curve}).
	\item Define an interpolation curve $\lambda\mapsto \bS(\lambda)$ between matrices $\bS_p^{(1)},\dots,\bS_p^{(N)}$, using curves obtained by~\eqref{eq:Spatial_Time_Curve}.
\end{enumerate}
We now detail two key points: the first concerns a new type of SVD, called oriented SVD, which allows defining the matrices $\bPhi_{p}^{(k)}$ and $\bPsi_{p}^{(k)}$ in a unique way. Finally, we will explain in \autoref{subsec:Space_Time_Algo} how to construct the curve $\lambda\mapsto \bS(\lambda)$, which requires the introduction of a mixed part.

%% ----------------------------------------------
\subsection{Oriented SVD on generic matrices}
%% ----------------------------------------------

As already noticed in~\autoref{sec:POD}, any computation of a POD of mode $p$ of a matrix $\bS\in \VecMat{n}{m}$ can be obtained from a SVD. Suppose now that $\bS$ is of rank $r\geq p$. Any SVD of $\bS$ with singular values $\sigma_{1}>\dotsc>\sigma_{r}$ leads to \emph{spatial orthonormal vectors} $\phi_{1},\dotsc,\phi_{r}$ in $\RR^{n}$ (the left singular vectors) and \emph{temporal orthonormal vectors} $\psi_{1},\dotsc,\psi_{r}$ in $\RR^{m}$ (the right singular vectors). A POD of mode $p$ then writes
\begin{equation}\label{eq:POD_mode_p}
	\bS_p=\bPhi_p \bSig_p \bPsi_p^T,\quad \bPhi_p:=[\phi_1,\dots,\phi_p],\quad \bSig_p:=\text{diag}(\sigma_{1},\dotsc,\sigma_{p}),\quad \bPsi_p:=[\psi_1,\dots,\psi_p].
\end{equation}
Now, because of sign indeterminacy of the spatial vectors $\phi_i$ and temporal vectors $\psi_i$, the matrices $\bPhi_p,\bPsi_p$ are not uniquely defined.

To overcome this difficulty, we need to introduce a \emph{new SVD} so that, under the assumption of genericity (see Definition~\ref{def:Genericity}), the matrices $\bPhi_p$ and $\bPsi_p$ given by~\eqref{eq:POD_mode_p} can be well-defined.   

The main idea of the new SVD introduced here is to make an intrinsic choice on the orientation for each space and temporal vector. Indeed, for each spatial vector $\phi$, only two choices can occur: $\phi$ or $-\phi$ (thus inducing a choice on the associated temporal vector). A choice of orientation is then made as follows. Taking the column vectors $\bS=[\bs_1,\dotsc,\bs_m]$ and $\bs$ to be the first column vector such that the scalar product $\langle \bs,\phi\rangle$ is non zero, we impose the sign taking $\langle \bs,\phi\rangle>0$.  

Let us now give all details to compute the oriented SVD before obtaining algorithm~\ref{alg:compact_oriented_svd}. A first Lemma, obtained by direct computation, allows us to use a column vector of the initial snapshot matrix $\bS$ to choose orientation: 
\begin{lemma}
	Let us consider $\bs_{1},\dotsc,\bs_{m}\in \RR^{n}$ to be the column vectors of $\bS\in \VecMat{n}{m}$ and take $\phi\in \RR^{n}$ to be a unit spatial vector of $\bS$, associated with a non--zero singular value $\sigma$. Then, there exists $i\in \{ 1,\dotsc,m\}$ such that $\langle \bs_{i},\phi\rangle=\bs_{i}^{T} \phi\neq 0$. 
\end{lemma}
From this, for any unit spatial vector $\phi\in \RR^{n}$ of $\bS$, let us define $\bs(\phi)$ to be the first column vector $\bs_i$ in $\bS=[\bs_1,\dotsc,\bs_n]$ such that $\langle\phi,\bs_i\rangle\neq 0$:
\begin{equation}\label{eq:Def_Associated_SVector}
	\bs(\phi):=\bs_i,\quad i:=\min\left\{j,\quad \langle \bs_{j},\phi\rangle \neq 0 \right\}.
\end{equation}

Any spatial eigenvector can therefore have a specific orientation:
\begin{definition}[Oriented eigenvectors]
	Let $\bS\in \mathrm{Mat}_{n,m}(\RR)$ and $\phi\in \RR^{n}$ a unit spatial vector associated to a non--zero singular value $\sigma$. Then $\phi$ is said to be \emph{oriented} if $\langle \bs(\phi),\phi\rangle>0$.
\end{definition}
From all this, let us now deduce the new SVD:
\begin{lemma}[Oriented SVD]\label{lem:SVD_sv_Generic}
	Let $\bS\in \mathrm{Mat}^{0}_{n,m}(\RR)$ ($m\leq n$) of rank $r$ such that all its non-zero singular values are distinct. Then, there exists one and only one couple of matrices 
	\begin{equation}
	\bPhi=[\phi_{1},\dotsc,\phi_{r}]\in \mathrm{Mat}_{n,r}(\RR),\quad \bPsi=[\psi_{1},\dotsc,\psi_{r}]\in \mathrm{Mat}_{m,r}(\RR)
	\end{equation}
	such that  
	\begin{equation}
	\langle \phi_{i},\phi_{j}\rangle=\langle \psi_{i},\psi_{j}\rangle=\delta_{ij},\quad \bS=\bPhi\bSig\bPsi^{T},\quad \bSig:=\text{Diag}(\sigma_1,\dotsc,\sigma_{r})\in \mathrm{Mat}_{r,r}(\RR)
	\end{equation}
	and $\phi_{i}$ are oriented spatial unit eigenvectors:
	\begin{equation}
	\langle \bs(\phi_{i}),\phi_{i}\rangle>0
	\end{equation}
	with $\bs(\phi_{i})$ defined by~\eqref{eq:Def_Associated_SVector}. Such a decomposition is called an \emph{oriented SVD}.
\end{lemma}

\begin{proof}
	First, any couple $(\phi,\psi)$ of spatial--temporal unit eigenvector for $\bS$ is defined modulo $\pm 1$, and $\psi$ is obtained in a unique way from $\phi$. 
	
	Let us suppose now we do not have uniqueness, so that there exist two unit spatial vectors $\phi$ and $\phi'$ associated to $\sigma$ such that
	\begin{linenomath} \begin{equation*}
	\langle \bs(\phi),\phi\rangle>0 \text{ and } \langle \bs(\phi'),\phi'\rangle>0.
	\end{equation*} \end{linenomath}
    We necessary have $\phi'=-\phi$ and $\bs(\phi)=\bs(\phi')$ so we deduce that
	\begin{linenomath} \begin{equation*}
	\langle \bs(\phi'),\phi'\rangle>=-\langle \bs(\phi),\phi\rangle>0
	\end{equation*} \end{linenomath}
	which is a contradiction, and we can conclude our proof.
\end{proof}

We give now an algorithm to obtain such an oriented SVD:

\begin{algo}[Oriented SVD]\label{alg:compact_oriented_svd}
	
		\begin{itemize}
			\item[$\bullet$] \textbf{Inputs}: $m\leq n$ and $\bS\in \mathrm{Mat}^{0}_{n,m}(\RR)$ of rank $r$.
			\item[$\bullet$] \textbf{Output}: Unique matrices $\bPhi$ and $\bPsi$ for an oriented SVD of $\bS$. 
		\end{itemize}
	
	\begin{enumerate}
		\item Compute a SVD of $\bS$ so that to obtain spatial unit vectors $\phi_{1},\dotsc,\phi_{r}$ and temporal unit vectors $\psi_{1},\dotsc,\psi_{r}$. 
		\item Consider the column vectors $\bs_{1},\dotsc,\bs_{m}$ of $\bS$.
		\item For $i=1,\dots,r$ define
		\begin{linenomath} \begin{equation*}
		\varepsilon_{i}:=\frac{\langle \bs(\phi_i),\phi_i\rangle}{\vert \langle \bs(\phi_i),\phi_i\rangle\vert}
		\end{equation*} \end{linenomath}
		where $\bs(\phi_i)$ is the first column vector $\bs$ of $\bS$ such that $\langle \phi_i,\bs\rangle\neq 0$, see~\eqref{eq:Def_Associated_SVector}.
		\item For $i=1,\dots,r$, make sign replacement
		\begin{linenomath} \begin{equation*}
		\phi_{i}\leftarrow \varepsilon_{i}\phi_{i},\quad \psi_{i}\leftarrow \varepsilon_{i}\psi_{i}.
		\end{equation*} \end{linenomath}
	\end{enumerate}
\end{algo}

\begin{example}
	Assume the rank 3 matrix
	\begin{linenomath} \begin{equation*}
		\bS=\begin{bmatrix}
		1 & 0 & 1 \\ 
		-1 & 1 & 0 \\
		0 & 2 & -1 \\
		0 & -1 & 0 \\
		1 & 0 & 1 \\
		0 & 0 & 0
		\end{bmatrix}=[\bs_1,\bs_2,\bs_3]
	\end{equation*} \end{linenomath}
	where a unit spatial vector corresponding to the largest singular value is given by (with 5 digits)
	\begin{linenomath} \begin{equation*}
		\phi_1=\begin{bmatrix}
			-0.31145 \\ 0.41763 \\ 0.74265 \\ -0.28294 \\ -0.31145
		\end{bmatrix}
	\end{equation*} \end{linenomath}
	and we can check that $\bs(\phi_1)=\bs_1$ with $\langle \phi_1,\bs_1\rangle<0$ so that we consider $-\phi_1$ instead of $\phi_1$, and so on.
\end{example}

%% ----------------------------------------------
\subsection{Space--Time interpolation algorithm}\label{subsec:Space_Time_Algo}
%% ----------------------------------------------

In this subsection, we define a Space--Time interpolation on any family of POD of mode $p$ taken from \emph{generic} snapshot matrices (see an overview in Figure~\ref{fig:Space_Time_Algo}). Such interpolation captures both the spatial and temporal part of such matrices, which is necessary from the point of view of mechanical equations, but we will also need to define a specific \emph{mixed part} of each POD (see lemma~\ref{lem:Mixed_Part_Interpol}).

Take back parameter values $\lambda_{1}<\dotsc <\lambda_{N}$, corresponding to snapshot matrices $\bS^{(1)},\dotsc,\bS^{(N)}$ in $\VecMat{n}{m}$, with $n=3N_s$ and $m=N_t$. To make use of the oriented SVD, let us suppose:

\textbf{Genericity assumption:} All snapshot matrices $\bS^{(1)},\dotsc,\bS^{(N)}$ have distinct non--zero singular values.

Take now $p$ to be some integer (less or equal than the minimum rank of all matrices $\bS^{(k)}$). Using the oriented SVD given by Algorithm~\ref{alg:compact_oriented_svd}, we can consider a POD of mode $p$ on each matrix $\bS^{(k)}$:
\begin{equation}
\bS_p^{(k)}:=\bPhi_{p}^{(k)} \bSig_{p}^{(k)}{\bPsi_{p}^{(k)}}^{T}\in \VecMat{n}{m}
\end{equation}
where $\bSig_{k}$ corresponds to singular values, and $\bPhi_{p}^{(k)}$ as well as $\bPsi_{p}^{(k)}$ uniquely define points in a compact Stiefel manifold:
\begin{equation}
	\bPhi_{p}^{(k)}:=[\phi_{1}^{(k)},\dotsc,\phi_{p}^{(k)}]\in \StieComp{p}{n},\quad \bPsi_{p}^{(k)}:=[\psi_{1}^{(k)},\dotsc,\psi_{p}^{(k)}]\in \StieComp{p}{m},
\end{equation}
Recall that in previous equation, $\phi_{1}^{(k)},\dotsc,\phi_{p}^{(k)}$ (resp. $\psi_{1}^{(k)},\dotsc,\psi_{p}^{(k)}$) correspond to spatial oriented eigenvectors (resp. temporal ones) of $\bS^{(k)}$.

Using the target Algorithm~\ref{alg:Target_Alg} first for the spatial matrices $\bPhi_{p}^{(k)}$ and then for the temporal matrices $\bPsi_{p}^{(k)}$, we obtain two curves
\begin{linenomath} \begin{equation*}
	\lambda \mapsto \bPhi(\lambda),\quad \lambda \mapsto \bPsi(\lambda).
\end{equation*} \end{linenomath}
Now, our goal is to produce an interpolation curve between the matrices $\bS_p^{(k)}$, taking into account both spatial and temporal curves defined above. Such a curve is given by
\begin{equation}\label{eq:Interpol_Curve_ST}
	\lambda \mapsto \bS(\lambda):=\bPhi(\lambda)\bM(\lambda)\bPsi(\lambda)^T \text{ with } \bS(\lambda_{k})=\bS_p^{(k)}.
\end{equation}
\begin{lemma}\label{lem:Mixed_Part_Interpol}
	For a curve defined~\eqref{eq:Interpol_Curve_ST} to be an interpolation curve between the matrices $\bS_p^{(k)}$, we necessary have
	\begin{equation}\label{eq:Formula_Mixed_Part}
		\bM(\lambda_{k})=\bPhi(\lambda_{k})^T\bS_{p}^{(k)}\bPsi(\lambda_{k}).
	\end{equation} 
\end{lemma}
\begin{proof}
	To satisfy~\eqref{eq:Interpol_Curve_ST}, we necessary have
	\begin{equation}\label{eq:Rel_for_lemma_Mixed_part}
		\bS(\lambda_{k})=\bPhi(\lambda_k)\bM(\lambda_k)\bPsi(\lambda_k)^T=\bS_p^{(k)},
	\end{equation}
	where, by construction, we have $\bPhi(\lambda_k)\in \StieComp{p}{n}$ and $\bPsi(\lambda_k)\in \StieComp{p}{m}$ (see target Algorithm~\ref{alg:Target_Alg}). We thus have
	\begin{linenomath} \begin{equation*}
		\bPhi(\lambda_{k})^T\bPhi(\lambda_k)=\bPsi(\lambda_{k})^T\bPsi(\lambda_k)=\mathbf{I}_p,
	\end{equation*} \end{linenomath}
	so that by the left and right multiplication of~\eqref{eq:Rel_for_lemma_Mixed_part} we obtain formula~\eqref{eq:Formula_Mixed_Part} for the mixed part. 
\end{proof}

Now we have:
\begin{lemma}
Let $\lambda\mapsto \bM(\lambda)\in \VecMat{p}{p}$ be any interpolated curve between the mixed part matrices 
\begin{linenomath} \begin{equation*}
	\bM_k:=\bPhi(\lambda_{k})^T\bS_{p}^{(k)}\bPsi(\lambda_{k})\in \VecMat{p}{p}
\end{equation*} \end{linenomath}
so that $\bM(\lambda_k)=\bM_k$ for $k=1,\dotsc,N$. Then, using the curve $\lambda\mapsto \bPhi(\lambda)$ (resp. $\lambda\mapsto \bPsi(\lambda)$) defined by the target Algorithm~\ref{alg:Target_Alg} applied on the matrices $\bPhi_{p}^{(k)}$ (resp. $\bPsi_{p}^{(k)}$), the curve
\begin{equation}
\kappa\: : \: \lambda\mapsto \bPhi(\lambda)\bM(\lambda)\bPsi(\lambda)^{T}
\end{equation}
is an interpolated curve between the matrices $\bS_p^{(1)},\dotsc,\bS_p^{(N)}$, so that $\kappa(\lambda_k)=\bS_p^{(k)}$ for each $k=1,\dotsc,N$.
\end{lemma}

\begin{proof}
	We need to check that $\kappa(\lambda_k)=\bS_p^{(k)}$ for each $k=1,\dotsc,N$. Now:
\begin{align*}
\kappa(\lambda_k)&=\bPhi(\lambda_{k})\bM(\lambda_k)\bPsi(\lambda_{k})^{T}=
\bPhi(\lambda_{k})\bM_k\bPsi(\lambda_{k})^{T} \\
&=\bPhi(\lambda_{k})\bPhi(\lambda_{k})^T\bS_{p}^{(k)}\bPsi(\lambda_{k})\bPsi(\lambda_{k})^{T} \\
&=\bPhi(\lambda_{k})\bPhi(\lambda_{k})^T\underbrace{\bPhi_{p}^{(k)} \bSig_{p}^{(k)}{(\bPsi_{p}^{(k)})}^{T}}_{\bS_{p}^{(k)}}\bPsi(\lambda_{k})\bPsi(\lambda_{k})^{T}
\end{align*}
where $\bPhi(\lambda_{k})\bPhi^T(\lambda_{k})$ corresponds to the projection matrix on the subspace $\bom_{k}:=\pi\left(\bPhi(\lambda_{k})\right)=\pi\left(\bPhi_{k}\right)$ (see Remark~\ref{rem:Target_Not_Interpolated_Curve}) so that
\begin{equation}
\bPhi(\lambda_{k})\bPhi^T(\lambda_{k})\bPhi_{p}^{(k)}=\bPhi_{p}^{(k)}
\end{equation}
and the same being true for the temporal part, we obtain the proof of the lemma. 
\end{proof}

The Space--Time interpolation algorithm is now given by:
\begin{algo}[Space--Time interpolation]\label{alg:Space_time_interpolation}

	\begin{itemize}
		\item[$\bullet$] \textbf{Inputs}:
		\begin{itemize}
			\item Generic matrices $\bS^{(1)},\dotsc,\bS^{(N)}$ in $\VecMatG{n}{m}$ ($m\leq n$), corresponding to parameter values $\lambda_{1}<\dotsc <\lambda_{N}$.
			\item A reference parameter value $\lambda_{i_0}$ with $i_0\in \{1,\dots,N\}$.
			\item A mode $p\leq m$. 
			\item A  parameter value $\widetilde{\lambda}$. 
		\end{itemize}
		\item[$\bullet$] \textbf{Output}: A matrix $\widetilde{\bS}\in \VecMat{n}{m}$. 
	\end{itemize}

\begin{enumerate}
\item Compute an oriented SVD on each matrix $\bS^{(k)}$ and write a POD of mode $p$ 
\begin{linenomath} \begin{equation*}
	\bS_p^{(k)}:=\bPhi_{p} \bSig_p^{(k)}(\bPsi_p^{(k)})^{T}\in \VecMat{n}{m}
\end{equation*} \end{linenomath}
with $\bPhi_p^{(k)}\in \StieComp{p}{n}$ and $\bPsi_p^{(k)}\in \StieComp{p}{m}$ uniquely defined.
\item Consider the target Algorithm~\ref{alg:Target_Alg} applied to 
the spatial parts $\bPhi_p^{(1)},\dotsc,\bPhi_{p}^{(N)}$, reference parameter value $\lambda_{i_0}$ and each of the $N+1$ parameter values $\lambda_{1},\dots,\lambda_{N},\wtlambda$, so from~\eqref{eq:Y_lambda} we can define matrices in $\StieComp{p}{n}$:
\begin{equation}\label{eq:Spatial_part}
	\bPhi(\lambda_{k}):=\bY(\lambda_{k}),\quad \bPhi(\wtlambda):=\bY(\wtlambda).
\end{equation}
\item Consider the target Algorithm~\ref{alg:Target_Alg} applied to the temporal
parts $\bPsi_p^{(1)},\dotsc,\bPsi_{p}^{(N)}$ reference parameter value $\lambda_{i_0}$ and each of the $N+1$ parameter values $\lambda_{1},\dots,\lambda_{N},\wtlambda$, so from~\eqref{eq:Y_lambda} we can define matrices in $\StieComp{p}{m}$:
\begin{equation}\label{eq:Time_part}
\bPsi(\lambda_{k}):=\bY(\lambda_{k}),\quad \bPsi(\wtlambda):=\bY(\wtlambda).
\end{equation}
\item For each $k=1,\dotsc,N$, define the square matrix of the mixed part
\begin{equation}\label{eq:Mixed_part}
\bM_{k}:=\bPhi(\lambda_{k})^T\bS_p^{(k)}\bPsi(\lambda_{k})\in \VecMat{p}{p}.
\end{equation}
\item Use a standard interpolation on square matrices $\bM_1,\dotsc\bM_N$, for instance:
\begin{equation}\label{eq:Coupled_part}
\bM(\wtlambda):=\sum_{i=1}^{N} \prod_{i \neq j} \frac{\wtlambda-\lambda_j}{\lambda_i-\lambda_j}\bM_i
\end{equation}
\item Using the spatial part $\bPhi(\wtlambda)\in \StieComp{p}{n}$ from~\eqref{eq:Spatial_part}, the temporal part $\bPsi(\wtlambda)\in \StieComp{p}{m}$ from~\eqref{eq:Time_part}, and the mixed part $\bM(\wtlambda)\in \VecMat{p}{p}$ from~\eqref{eq:Coupled_part}, the interpolated snapshot matrix corresponding to $\wtlambda$ is finally given by 
\begin{linenomath} \begin{equation*}
\widetilde{\bS}:=\bPhi(\wtlambda)\bM(\wtlambda)\bPsi(\wtlambda)^T\in \VecMat{n}{m}.
\end{equation*} \end{linenomath}
\end{enumerate}
\end{algo}

\section{Rigid-Viscoplastic FEM Formulation}
\label{sec:Rigid-Viscoplastic_FEM}

The main defining characteristic of the RVP formulation is that it neglects the elasticity effects. This \emph{idealization} is based on the fact that elastic components of strain remain small as compared with irreversible strains. This means that the additive decomposition of the total strain-rate tensor $\dot{\varepsilon}_{ij}=\dot{\varepsilon}^{e}_{ij}+\dot{\varepsilon}^{p}_{ij}$ simplifies to $\dot{\varepsilon}_{ij}=\dot{\varepsilon}^{p}_{ij}$, where $\dot{\varepsilon}^{e}_{ij}$ is the elastic component of the strain-rate tensor, $\dot{\varepsilon}^{p}_{ij}$ is the plastic component and $\dot{\varepsilon}_{ij}$ is the total strain-rate tensor.  Therefore, the RVP formulation turns out to be very similar to fluid flow problems\textcolor{blue}{,} and it is also called  \emph{flow formulation}~\cite{Hill1998}. Although it is not possible to calculate the residual stresses and the spring-back effect, the flow formulation presents several advantages. Unlike the elastoplastic FEM, the RVP formulation, even though more approximate, is more stable, simpler to be implemented in computer codes, and can use relatively larger time increments, thus improving the computational efficiency. A thorough overview of the foundation of the theory can be found in~\cite{Kobayashi1989,Chenot1992}.

\subsection{Governing Field Equations}

Classical rigid viscoplastic problems consider the plastic deformation of an isotropic body occupying a domain $\Omega \subset \mathbb{R}^{3}$. The domain $\Omega$ and its boundary $\partial \Omega$ represent the current configuration of a body according to the \emph{Updated Lagrangian} formulation. The governing equations that have to be satisfied are:

\noindent (a) Equilibrium condition:
\begin{linenomath} \begin{equation*}
\qquad\sigma_{ij,j}=0
\end{equation*} \end{linenomath}
\noindent (b) Compatibility conditions:
\begin{linenomath} \begin{equation*}
\qquad\dot{\varepsilon}_{ij}=\frac{1}{2}(v_{i,j}+v_{j,i})
\end{equation*} \end{linenomath}
\noindent (c) Yield criterion:
\begin{linenomath} \begin{equation*}
\bar{\sigma}:=\bigg(\frac{2}{3}\sigma'_{ij}\sigma'_{ij}\bigg)^{\frac{1}{2}} = \bar{\sigma}(\bar{{\varepsilon}},\dot{\bar{\varepsilon}},T)
\end{equation*} \end{linenomath}
\noindent (d) Constitutive equations:
\begin{equation}
\sigma'_{ij}=\frac{2}{3}\frac{\bar{\sigma}}{\dot{\bar{\varepsilon}}}\dot{\varepsilon}_{ij}, \qquad
\dot{\bar{\varepsilon}}=\bigg(\frac{2}{3}\dot{\varepsilon}_{ij}\dot{\varepsilon}_{ij}\bigg)^{\frac{1}{2}}
\label{eq:Constitutive_equations}
\end{equation}

\noindent (e) Incompressibility condition:	
\begin{linenomath} \begin{equation*}
\dot{\varepsilon}_{v}:=\dot{\varepsilon}_{kk}=0
\end{equation*} \end{linenomath}
\noindent (f) Boundary conditions:	
\begin{align*}
\textbf{v}=\hat{\textbf{v}}\qquad &\text{on}\qquad \partial \Omega_{v} \\
\textbf{F}=\hat{\textbf{F}}\qquad &\text{on}\qquad \partial \Omega_{F} \\
\text{friction and contact} \qquad &\text{on}\qquad \partial \Omega_{c}
\end{align*}
In the above equations $\pmb\sigma=(\sigma_{ij})$ is the stress tensor, $\dot{\pmb\varepsilon}=(\dot{\varepsilon}_{ij})$ is the strain rate tensor, $v_{i}$ are velocity components, $\bar{\sigma}$ is the effective stress, $\dot{\bar{\varepsilon}}$ is the second invariant of $\dot{\pmb\varepsilon}$ called effective strain rate, and $\pmb\sigma'=(\sigma'_{ij})$ is the deviatoric stress tensor defined by $\sigma'_{ij}=\sigma_{ij}-\delta_{ij}\sigma_{kk}/3$.

The hat symbol ~$\hat{}$~ denotes prescribed values. Generally, the boundary $\partial \Omega$ consists of three distinct parts: over $\partial \Omega_{v}$ velocity conditions are prescribed (essential boundary conditions), $\partial \Omega_{F}$ is the part where the traction conditions are imposed in the form of nodal point forces (natural boundary conditions), while the boundary conditions along $\partial \Omega_{c}$ are mixed, and neither the velocity nor the force can be described. Therefore, we have the disjoint union:

\begin{equation}
\partial \Omega=\partial \Omega_{v}\cup \partial \Omega_{F}\cup \partial \Omega_{c}
\end{equation}

\subsection{Variational form}

In a variational formulation, the functional $\Pi$ (energy rate) is defined by an integral form in accordance with the virtual work-rate principle

\begin{equation}
\Pi(v):=\int_{\Omega}\bar{\sigma}\dot{\bar{\varepsilon}}dV-\int_{\partial \Omega_{F}}F_{i}v_{i} dS
\label{eq:functional_equation}
\end{equation}

where the first term in~(\ref{eq:functional_equation}) represents the internal deformation work-rate, whereas the second term represents the work-rate done by the external forces. $F_{i}$ denotes prescribed surface tractions on the boundary surface $\partial\Omega_{F}$. Recalling the MarKov-Hill~\cite{Markov1948,Hill1948} variational principle, among all virtual (admissible) continuous and continuously differentiable velocity fields $v_{i}$ satisfying the conditions of compatibility and incompressibility, as well as the velocity boundary conditions, the real velocity field gives to the functional $\Pi$ a stationary value, i.e., the first-order variation vanishes. Moreover, in order to relax the incompressibility constraint condition $\dot{\varepsilon}_{v}=\dot{\varepsilon}_{kk}=0$ on an admissible velocity field, a classical penalized form is used

\begin{equation}
\delta \Pi:=\int_{\Omega}\bar{\sigma}\delta\dot{\bar{\varepsilon}}dV+\frac{1}{2}\int_{\Omega}K\dot{\varepsilon}_{v}\delta\dot{\varepsilon}_{v} dV-\int_{\partial \Omega_F} F_{i} \delta v_{i} dS = 0
\label{eq:functional_equation_penalized}
\end{equation}

where $K$ is a large positive constant which penalizes the dilatational strain-rate component. It can be shown that the mean stress is $\sigma_{m}=K\dot{\varepsilon}_{kk}$. 

\begin{remark}
A limitation of the Updated Lagrangian method for large deformation problems is the excessive element distortion. To this end, remeshing processes are necessary to simulate unconstrained plastic flows. A mesh generation process is activated in case of zero or negative determinant of the Jacobian matrix, or due to various element quality criteria. Then, a new mesh is calculated conforming to the current state of the geometry followed by an interpolation of the state variables between the old and the newly generated mesh.
Thus, the information of the remapping process has to adequately be transferred to the ROM basis obtained using the POD snapshot method. We remark that at this first attempt, we avoid remeshings of the workpiece during the course of the simulation. This topic will be addressed in a future investigation.
\end{remark}

\subsection{Discretization and iteration}

The discretization of the functional follows the standard procedure of the finite element method. Eq.~(\ref{eq:functional_equation_penalized}) is expressed in terms of nodal point velocities $v_{i}$ and their variations $\delta v_{i}$. Using the variational principle
\begin{equation}
\delta \Pi=\sum^{M}_{m=1}\frac{\partial\Pi^{(m)}}{\partial v_{i}} \delta v_{i}=0,\qquad i=1,2,...,2N_s,
\end{equation}
where $\delta v_{i}$ are arbitrary except that they must be zero to satisfy the corresponding essential boundary conditions, and $M$ denotes the number of elements. From the arbitrariness of $\delta v_{i}$, a set of algebraic equations (stiffness equations) are obtained

\begin{equation}
\frac{\partial\Pi}{\partial v_{i}}=\sum^{M}_{m=1}\frac{\partial\Pi^{(m)}}{\partial v_{i}}=0.
\label{eq:Nonlinear_algebric_equations}
\end{equation}

As the resulting algebraic equations are highly nonlinear, they linearized by the Taylor expansion near an assumed velocity field $\textbf{v} =\textbf{v}_{0}$ as

\begin{equation}
\frac{\partial\Pi}{\partial v_{i}} \Bigg|_{\textbf{v}=\textbf{v}_{0}}+ \frac{\partial^{2}\Pi}{\partial v_{i}\partial v_{j}}\Bigg|_{\textbf{v}=\textbf{v}_{0}}\Delta v_{j}=0
\label{eq:Linearization}
\end{equation}

where the first factor of the second term is also known as the Jacobian of the system (Hessian matrix), and $\Delta v_{j}$ is a first-order correction of the velocity component $v_{j}$. Solving~(\ref{eq:Linearization}) with respect to $\Delta v_{j}$, the assumed velocity field is updated by the form (written in vector notation)

\begin{equation}
\textbf{v}^{(i)}=\textbf{v}^{(i-1)} + \alpha (\Delta \textbf{v})^{(i)}
\end{equation}

where $0\leq\alpha\leq 1$ and $i$ is the  iteration step. The solution is obtained  by the Direct iteration method~\cite{Kobayashi1989,Oh1982} and/or by Newton-Raphson type methods. The iteration process is repeated until the following described convergence criteria are satisfied simultaneously
\begin{equation}
\frac{\parallel \Delta \textbf{v} \parallel_{L_2}}{\parallel\textbf{v}\parallel_{L_2}}\leq e_{1}, \qquad
\left\|\frac{\partial\Pi}{\partial \textbf{v}}\right\|_{L_2}\leq e_{2}
\label{eq:Error_norms}
\end{equation}

namely\textcolor{blue}{,} the velocity error norm and the norm of the residual equations, where $e_{1}$ and $e_{2}$ are sufficiently small specified tolerance numbers. 

\subsection{Heat Transfer Analysis}

In the present model, a thermodynamically sound derivation is adopted using the conservation of energy

\begin{equation}
-\rho c \frac{\partial T}{\partial t} + k \nabla^{2}T + \xi \bar{\sigma}\dot{\bar{\varepsilon}} = 0 
\label{heat_equation}
\end{equation}

where $\rho c$ is the volume-specific heat of the material, $\xi \bar{\sigma}\dot{\bar{\varepsilon}}$ represents the work heat rate per unit volume due to plastic deformation, $k$ is the thermal conductivity, $T$ is the temperature and $\xi$ is a coefficient that presents the fraction of the deformation energy dissipated into heat also known as the Taylor-Quinney coefficient.

In a weak form, and using the divergence theorem
\begin{equation}
- \int_{\Omega} \xi \bar{\sigma}\dot{\bar{\varepsilon}} \delta T dV + 
\int_{\Omega}  k \nabla T \delta (\nabla T) dV +
\int_{\Omega}  \rho c \frac{\vartheta T}{\vartheta t} \delta T dV - \int_{\partial \Omega} q_n  \delta T dS = 0
\label{heat_divergence}
\end{equation}

where 

\begin{equation}
q_n := k \frac{\partial T}{\partial n}
\end{equation}

is the heat flux across the boundary $\partial \Omega$ and $n$ denotes the unit normal vector to the boundary surface $\partial \Omega$.

In standard finite element books, e.g.~\cite{Zienkiewicz1974}, it can be seen that the heat balance equations such as~\eqref{heat_divergence}, upon finite element discretization are reduced to the form:

\begin{equation}
\bm{C} \dot{\bm{T}} + \bm{K} \bm{T} = \bm{Q}
\label{heat_matrix_form}
\end{equation}

where $\bm{C}$ is the heat capacity matrix, $\bm{K}$ denotes the heat conduction matrix, $\bm{Q}$ is the heat flux vector, $\bm{T}$ is the vector of nodal point temperatures, and $\dot{\bm{T}}$ is the rate of temperature increase vector of nodal points. 

The theory necessary to integrate~\eqref{heat_matrix_form} can be found in numerical analysis books~\cite{Ralston2001,Dahlquist2008}. It suffices to say that one-step time integration is used. The convergence of a scheme requires consistency and stability. Consistency is satisfied by a general time integration scheme

\begin{equation}
^{t+\Delta t}\bm{T} = ^{t}\bm{T} + \Delta t [(1- \theta) ^{t}\bm{\dot{T}} + \theta ^{t + \Delta t}\bm{\dot{T}}]
\label{eq:difference_schemes}
\end{equation}

where $\theta$ is a parameter varying between 0 and 1 ($\theta = 0$: Forward difference, $\theta = 1/2$: Crank-Nicholson, $\theta = 2/3$: Galerkin, $\theta = 1$: Backward difference).

\begin{remark}
\emph{Unconditional stability} is obtained for  $\theta \geq 0.5$. This is important, because it is desirable to take time steps as large as the deformation formulation allows, since this is the most expensive part of the process.
\end{remark}

\subsection{Computational Procedure for Thermo-Mechanical Analysis}

For solving \emph{coupled} thermomechanical problems, two different approaches can be used. In the traditional \emph{monolithic} approach, a single solver is in charge of the solution of the entire system of equations. In an alternative approach, the mechanical and thermal solvers deal respectively with the viscoplastic flow and the thermal field equations. Thus, in the so-called \emph{staggered solution} procedure used here, the state of the system is advanced by sequentially executing and exchange information between these two solvers~\cite{Felippa1980}. The equations for the mechanical analysis and the temperature calculation are \emph{strongly coupled}, thereby making necessary the simultaneous solution of the finite element counterparts~\cite{Kobayashi1989,Rebelo1980,Rebelo1980a}.

\section{Numerical Investigations}
\label{sec:Numerical_Investigations}

The purpose of this section is to evaluate the performance of the ST POD interpolation using the  velocity and temperature fields during the course of the simulation of the forming process. As a benchmark test case, a rectangular cross-section bar is compressed between two parallel flat dies under the condition of a constant shear friction factor $m$ at the die-workpiece interface. The initial workpiece has dimensions $h=20$ mm (height) and $w=20$ mm (width). Plane strain conditions are considered. Due to the symmetry of the problem, only one quarter of the cross-section is analyzed. The velocity of the upper and the lower die is set to $v=1$ mm/s. The initial temperature of the die and the workpiece is set to $T=25$ $^{\circ}$C. The bar is compressed until a 35\% reduction in height is achieved. The final simulation state is accomplished in 7-time steps with a constant time increment $\Delta t = 0.5$ s. One can observe the complexity of the nonuniform deformation  presented by the barreling of the free surface (Figure~\ref{fig:Deformation_patterns}). 
In our calculations, we employ a conventional rate-dependent power law to describe the material flow stress equation
\begin{equation}
\bar{\sigma}(\dot{\bar{\varepsilon}}) =1000\dot{\bar{\varepsilon}}^{0.1}\qquad \text{(MPa)}
\end{equation}	

The solution convergence is assumed when the velocity error norm and the force error norm (\ref{eq:Error_norms}) becomes less than $10^{-6}$. The type of element used is the linear isoparametric rectangular element with four-point integration. However, one point integration is used for the dilatation term, the second integral of the functional in~(\ref{eq:functional_equation_penalized}). This is known as the reduced integration scheme which imposes the volume constancy averaged over the linear rectangular element. The computational grid composed of 100 elements interconnected at $N_s=121$ nodes with 2 degrees of freedom, resulting in a global stiffness matrix of size 242$\times$242. For the rigid-viscoplastic analysis, the limiting strain rate $\dot{\bar{\varepsilon}}_{0}$ is chosen to be 0.01 and the penalty constant (or bulk modulus) $K$ is set to $10^{5}$.

Among the various models of friction, the one proposed in~\cite{Chen1978} is adapted to model the sliding contact at the tool-workpiece interface. This model allows the variation of the tangential traction with the relative velocity at the tool-workpiece interface
\begin{linenomath} \begin{equation*}
\textbf{t}_{f}= -mk \frac{\textbf{v}_{s}}{|\textbf{v}_{s}|} \simeq -mk \Bigg\{\frac{2}{\pi}\arctan\Bigg(\frac{|\textbf{v}_{s}|}{v_{0}}\Bigg)\Bigg\} \frac{\textbf{v}_{s}}{|\textbf{v}_{s}|}
\end{equation*} \end{linenomath}	
where $\textbf{v}_{s}$ is the relative velocity in the tangential direction between the tool and the workpiece, and $v_{0}$ is a positive constant several orders of magnitude smaller than $\textbf{v}_{s}$; $m$ is the friction factor $(0 < m < 1)$  and $k$ is the material shear yield stress $k=\bar{\sigma}/\sqrt{3}$. For the compression tests considered here, the relative tangential velocity at the tool-workpiece interface at the beginning of deformation is zero. The present analysis assumes that the friction factor remains constant throughout compression. Investigations on frictional shear stress measurements over the interface between a cylindrical workpiece and a die during plastic compression are reported in~\cite{Rooyen1960}. The basic characteristics of algorithms used in the RVP FEM analysis are summarized in Table~\ref{table:Numerical_algorithms}.

\small{
\begin{table}[H]
\begin{tabular}{p{4.0cm} p{8.0cm} }
%\hline
\multicolumn{2}{p{13.2cm}}{\textbf{Basic characteristics of algorithms in RVP FEM}} \\
%\hline
Type of problem & Two dimensional, plane strain, rigid viscoplastic material flow, isotropic, homogeneous \\
\hline
Thermomechanical problem solution & Loose coupling (staggered) - Backward Euler difference ($\theta = 1$)\\
\hline
Type of elements & 4-node quadrilateral isoparametric elements, bilinear shape functions \\
\hline
Flow stress equation & Power law: $\bar{\sigma}(\dot{\bar{\varepsilon}}) =c\dot{\bar{\varepsilon}}^{p}, \qquad c,p$ constants\\
\hline
Iteration method & Direct, BFGS with line search\\
\hline
Remeshing & N/A \\
\hline
Boundary conditions & Sliding friction on $S_{c}$\\
%\hline
\end{tabular}
\caption{Numerical algorithms.}
\label{table:Numerical_algorithms}
\end{table}
}

\begin{remark}
Note that during the course of the simulation we avoid remeshing of the workpiece. As discussed in~\cite{Ryckelynck2009}, remeshing techniques can be taken into account provided that mesh transfer operations are applied to the reduced-basis.  
\end{remark}

\subsection{Mechanical field}

The first case for numerical illustration of the method considers the velocity field during the simulation of the forming process using the shear friction factor $m$ as the investigated parameter. From now on, let the parametric points corresponding to the shear friction factor $m$ denoted with  $\lambda$ for convenience with the previous sections. For the numerical study, the following training points are selected $\lambda \in \Lambda_t = \{0.1, 0.5, 0.9 \}$. The choice made here, is to use a minimum number of sampling points equi-distributed over the parametric range. The target point is set to $\wtlambda=0.3$. 
See the FEM solutions for the training  and target points at the final state of the computation in Figure~\ref{fig:Deformation_patterns}.

For each parametric simulation, a sequence of snapshots uniformly distributed over time using an increment of $\Delta t =0.5$ s is extracted for all nodes of the workpiece. The space-time snapshot matrices $\bS^{(i)} \in \VecMat{2N_s}{N_t}$ with $2N_s=242$ and $N_t=7$, corresponding to parameter values $\lambda_i$, are associated with the nodal velocity field in $x$ and $y$ directions.

For the parametric Space-Time interpolation, the  snapshot matrix  $\widetilde{\bS}$ of mode $p$ corresponding to the target point $\wtlambda$ is computed via the target Algorithm~\ref{alg:Space_time_interpolation}.  The target Algorithm~\ref{alg:Target_Alg} is applied to the spatial  $\bPhi_{p}^{(1)},\dotsc,\bPhi_{p}^{(N)}$ and temporal parts $\bPsi_{p}^{(1)},\dotsc,\bPsi_{p}^{(N)}$, with reference parameter value $\lambda_{i_0}=0.5$. In order to assess the interpolation acuracy, the snapshot matrix $\widetilde{\bS}$  is compared against the high-fidelity FEM solution by introducing the following a posteriori errors. Using the interpolated and the HF-FEM snapshot matrices $\widetilde{\bS}$ and $\bS^{\text{FEM}}$, respectively, the relative $L_2$-error measure is defined as
\begin{equation}
e_{L_2}(\widetilde{\bs}_i):= 
\frac{\Vert \widetilde{\bs}_i - \bs^{\text{FEM}}_i \Vert_{L_2}}{\Vert  \bs^{\text{FEM}}_i \Vert_{L_2}}, \quad
i=1,\dots,p \leq N_t.
\label{eq:L2_error_norm_HF}
\end{equation} 

Additionally, the relative Frobenius error norm of $\widetilde{\bS}$ and $\bS^{\text{FEM}}$ is defined as 
\begin{equation}
e_{F}(\widetilde{\bS}):= \Vert \mathbf{\widetilde{S}} - \mathbf{S}^{\text{FEM}} \Vert_{F} / \Vert \mathbf{S}^{\text{FEM}} \Vert_{F}.
\label{eq:Frobenious_error_norm}
\end{equation}

The eigenvalue spectrum of snapshot matrices $\bS^{(i)}$ corresponding to training points $\lambda_i \in \Lambda_t$ is exhibited in a semi-log scale in Figure~\ref{fig:Sing_values_magnitude}. We can observe that the distance between the first and the last eigenvalue is from 5 up to 6 orders of magnitude. Moreover, the percentage of energy $\mathcal{E}(k)= \sum_{i=1}^{k} \sigma^{2}_{i} / \sum_{i=1}^{N_t} \sigma^{2}_{i}$ captured from the POD modes is shown in Figure~\ref{fig:Energy_SVD}. It is evident that most of the $99.9 \%$ of the total energy is contained by the first two POD modes. 

The relative $L_2$-error norm $e_{L_2}(\widetilde{\bs}_i)$ (see~\eqref{eq:L2_error_norm_HF}) between the interpolated and the HF-FEM solution for various POD modes is displayed in Figure~\ref{fig:L2_norm_POD_modes}.
In general, the relative error for all POD modes lie within a range of 0.0175 up to 0.038. It can be observed that the interpolated ST POD solution delivers good accuracy and is reliable enough to predict the velocity field for the investigated target point.

\begin{remark}
In the case of using $p=7$ POD modes for the temporal basis interpolation, the Grassmannian manifold $\mathcal{G}(p,p)$ \emph{reduces to one point}, so it is not relevant to use the target Algorithm~\ref{alg:Target_Alg}: any new parameter value will give rise to the same matrix $\bPsi_{i_0}$ in the associated compact Stiefel manifold, corresponding to the reference point.
\end{remark}

Additionally, the position vector error $e_{L_2}(\mathbf{\widetilde{x}}(t)) = \Vert \mathbf{\widetilde{x}}(t)-\mathbf{x}^{\text{FEM}}(t) \Vert_{L_2}$ at the nodal points is computed for $p=$2,3,5 and 7 POD modes, where $\mathbf{\widetilde{x}}(t)$ and $\mathbf{x}^{\text{FEM}}(t)$ denotes the position vector of the ST POD and the high-fidelity FEM solutions, respectively, at the time increments during the deformation. Figure~\ref{fig:Snapshot_L2_norm_POD_modes_state_7} presents the local error $e_{L_2}(\mathbf{\widetilde{x}}(t))$ superimposed at the final loading state $t=0.35$ s obtained from the high-fidelity FEM solution. Different patterns of the spatial error distribution can be observed concerning the number of POD modes $p$. It is interesting to observe that in both cases, the maximum error is located near the upper-right location of the deforming workpiece. 

The evolution of the deformation process can be also represented using the time-displacement histories of some selected nodes of the workpiece (Figure~\ref{fig:Nodal_time_displacement_histories}). The ST POD predictions are compared against the high-fidelity FEM counterpart solution using $p=2$ POD modes. Again, it can be observed that the interpolated ST POD solution is accurate and reliable to predict the evolution of the displacement field for the investigated target point during the forming process. 

For the preceding numerical investigations, the ST POD efficiency is demonstrated using a single target point, i.e., $\widetilde{\lambda}=0.3$. To further assess the interpolation performance, a new target point is now considered,  $\widetilde{\lambda}=0.8$. Interpolation is performed using the same set of training points $\lambda \in \Lambda_t = \{0.1, 0.5, 0.9 \}$, with reference parameter value $\lambda_{i_0}=0.5$. The relative $L_2$-error norm $e_{L_2}(\widetilde{\bs}_i)$ for various POD modes $p$ corresponding to target point $\widetilde{\lambda}=0.8$ is shown in  Figure~\ref{fig:L2_norm_POD_modes_new_interp_target_m_08}. Again, one can observe that the relative error lies within a narrow range of the values, i.e., 0.014 up to 0.026.

%\begin{remark}
%We need to know what is the optimal choice in the sense of interpolation quality related to the Riemannian distances  (\ref{eq:Geodesic_distance_Grassmann}) between the training points and the applied interpolation method (Lagrange in our case). Or in other words, how the choice of the local chart, i.e., of the reference point $\bom_0$ on the Grassmannian, and the distances between it and the other training points affects interpolation accuracy. And furthermore, what is the size of the diameter of a `small' neighborhood around $\bom_0$? How do we arrange (sample) the training points on the Grassmann manifolds without any a priori knowledge of the underlying system dynamics? It would suffice to simply select uniform grids in case of multi-parametric problems? To the best of the \textcolor{blue}{author's knowledge, it remains an open problem of how all the above issues can be optimal chosen a priori}.
%\end{remark}

\subsection{Temperature field}

To further investigate the performance of the proposed ST POD interpolation, the temperature field obtained from the coupled thermomechanical simulation of the forming process is considered. Again, for the temperature field, we consider the shear friction factor $m$ as the investigated system parameter. The training points selected for the mechanical field analysis are also used in this study, i.e., $\lambda \in \Lambda_t = \{0.1, 0.5, 0.9 \}$. The target point is set to $\wtlambda=0.3$. For each parametric problem, snapshots are uniformly distributed over time using an increment step size  $\Delta t =0.5$ s. The final deformation state is reached at $t=0.35$ s. The space-time snapshot matrices  $\bS
^{(i)}\in \VecMat{N_s}{N_t} $ of size $121 \times 7$, corresponding to $\lambda_i$, are associated with nodal temperatures. We will now compare the Space-Time interpolation (see Algorithm~\ref{alg:Space_time_interpolation}) against the high-fidelity FEM solution. Again, for the target Algorithm~\ref{alg:Target_Alg} applied to the spatial  $\bPhi_{p}^{(1)},\dotsc,\bPhi_{p}^{(N)}$ and temporal parts $\bPsi_{p}^{(1)},\dotsc,\bPsi_{p}^{(N)}$, the reference parameter value $\lambda_{i_0}=0.5$ is used.

Figure~\ref{fig:Temperature_profiles} presents the temperature profiles at the final compression state obtained using different values of the shear friction factor $m$ (represented by parameter $\lambda$). The temperature rises due to plastic work conversion to heat assuming a constant value for the Taylor-Quinney coefficient $\xi = 0.9$. In all cases, the maximum temperature is located at the center of the workpiece with values ranging from $T=89.5$ $^{\circ}$C up to $T=98$ $^{\circ}$C.   

The eigenvalue spectrum of snapshot matrices $\bS^{(i)}$ corresponding to training points $\lambda_i \in \Lambda_t$ is shown in a semi-log scale in Figure~\ref{fig:Sing_values_magnitude_temperature}. We can observe that the distance between the first and the last eigenvalue of the curves is of the order of 5 up to 6 orders of magnitude. Moreover, the system energy $\mathcal{E}(k)= \sum_{i=1}^{k} \sigma^{2}_{i} / \sum_{i=1}^{N_t} \sigma^{2}_{i}$ captured from the POD modes is shown in Figure~\ref{fig:Energy_SVD_temperature}. Most of the $99.9 \%$ of the total energy is contained by the first two POD modes. 

The relative $L_2$-error norm $e_{L_2}(\widetilde{\bs}_i)$ \eqref{eq:L2_error_norm_HF} between the interpolated and the HF-FEM snapshot matrices $\widetilde{\bS}$ and $\bS^{\text{FEM}}$, respectively, for various modes $p$ is shown in Figure~\ref{fig:L2_norm_POD_modes_temperature}. Additionally, the Frobenius relative error norm  \eqref{eq:Frobenious_error_norm} for the POD modes is presented in Figure~\ref{fig:Frobenius_error_norm_temperature}. In general, the obtained results are found to have less than 1$\%$ relative error for POD modes $p>1$ and therefore are acceptable as fast near real-time numerical predictions. 

Finally, Figure~\ref{fig:Nodal_time_temperature_histories} shows the ST POD time-temperature histories for some selected nodes of the workpiece using $p=7$ modes. The predictions are compared against the high-fidelity counterpart solution,
and it is difficult to distinguish differences
among these plots. It is revealed that the interpolated ST POD solution delivers good accuracy for all selected nodes.

\subsection{Computational complexity}
The computational cost of the ST POD interpolation scales with the computational complexity of SVD and the matrix operations in the target ST Algorithm~\ref{alg:Space_time_interpolation}. It is evident, that the cost of ST POD interpolation will be lower compared to the standard POD Galerkin nonlinear approaches and even  lower than the full order FEM solution. The coupled thermomechanical FEM simulation for the target point takes 35.123 seconds in wall-clock time. On the other hand, the ST interpolation for the mechanical problem using a ROM POD basis of mode $p=4$ results in 0.147 seconds in wall-clock time. The ST interpolation for the thermal problem using a ROM POD basis of mode $p=4$ results in 0.153 seconds in wall-clock time. Therefore, the total ST interpolation takes 0.3 seconds in wall-clock time corresponding to a time speed-up of 116.96. All experiments in this section were implemented in Matlab and run on a 4th Generation Intel(R) Core(TM) i7-4600U CPU @ 2.10GHz, 8GB RAM, 250 GB SSD, Debian 9 x64.

%mechanical problem: 0.147408 s
%thermal problem:0.152858 
%FEM simulation: 35.123490
\section{Conclusions}
\label{sec:Conclusions}

A novel non-intrusive Space-Time POD basis interpolation scheme on compact Stiefel manifolds is developed and applied to parametric high nonlinear metal forming problems. Apart from the separate interpolation of POD spatial and temporal basis on associated Grassmannian manifolds, an interpolation function is defined on a set of parametric snapshot matrices. This function results from curves, which are defined on compact Stiefel manifolds both for space and the temporal part, and also the use of some mixed part encoded by a square matrix. This latter matrix provides a link between the interpolated space and temporal basis for the construction of the target ROM snapshot matrix. To prove the efficiency of the method it has been used a coupled thermomechanical rigid-viscoplastic FEM formulation which is integrated into the manufacturing industry in a variety of applications. The performed numerical investigations have considered the reconstruction of the ROM snapshot matrices both of the velocity and the temperature fields. Moreover, the error norms of the Space-Time POD interpolated ROM models concerning the  associated high-fidelity FEM counterpart solutions are validating the accuracy of the proposed interpolation scheme. In conclusion, the overall results demonstrate the potential use of the proposed ST POD interpolation scheme for near real-time parametric simulations using off-line computed ROM POD databases, supporting thus manufacturing industries to accelerate design-to-production timespans, and thereby reducing costs while ensuring the design of superior processes.

\vfill
\pagebreak

\begin{figure}[H]
\resizebox{0.95\textwidth}{!}{%	
\includegraphics{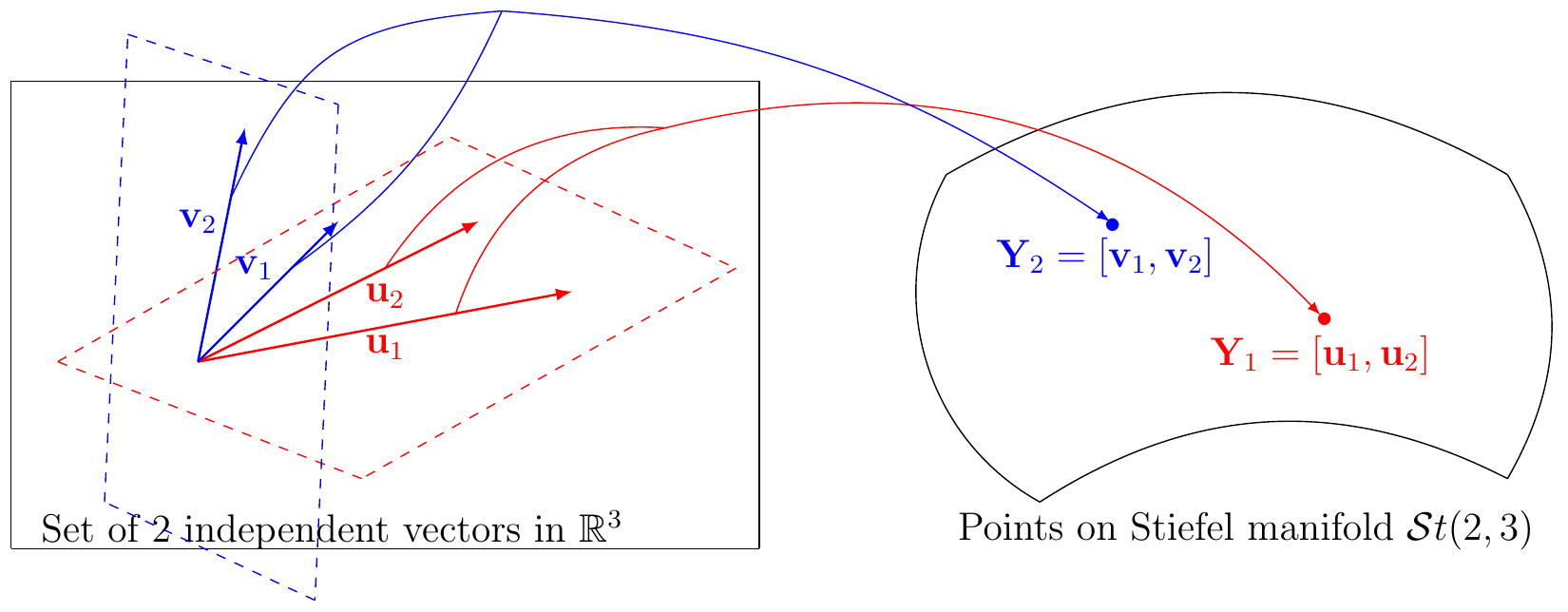}
}
\caption{Points on Stiefel manifold. The linearly independent vectors in $\mathbb{R}^3$ spanning the red and blue planes correspond to points in $\StieComp{2}{3}$.  
}
\label{fig:Two_Points_Stiefel}
\end{figure}

\begin{figure}[H]
\resizebox{0.97\textwidth}{!}{%	
\includegraphics{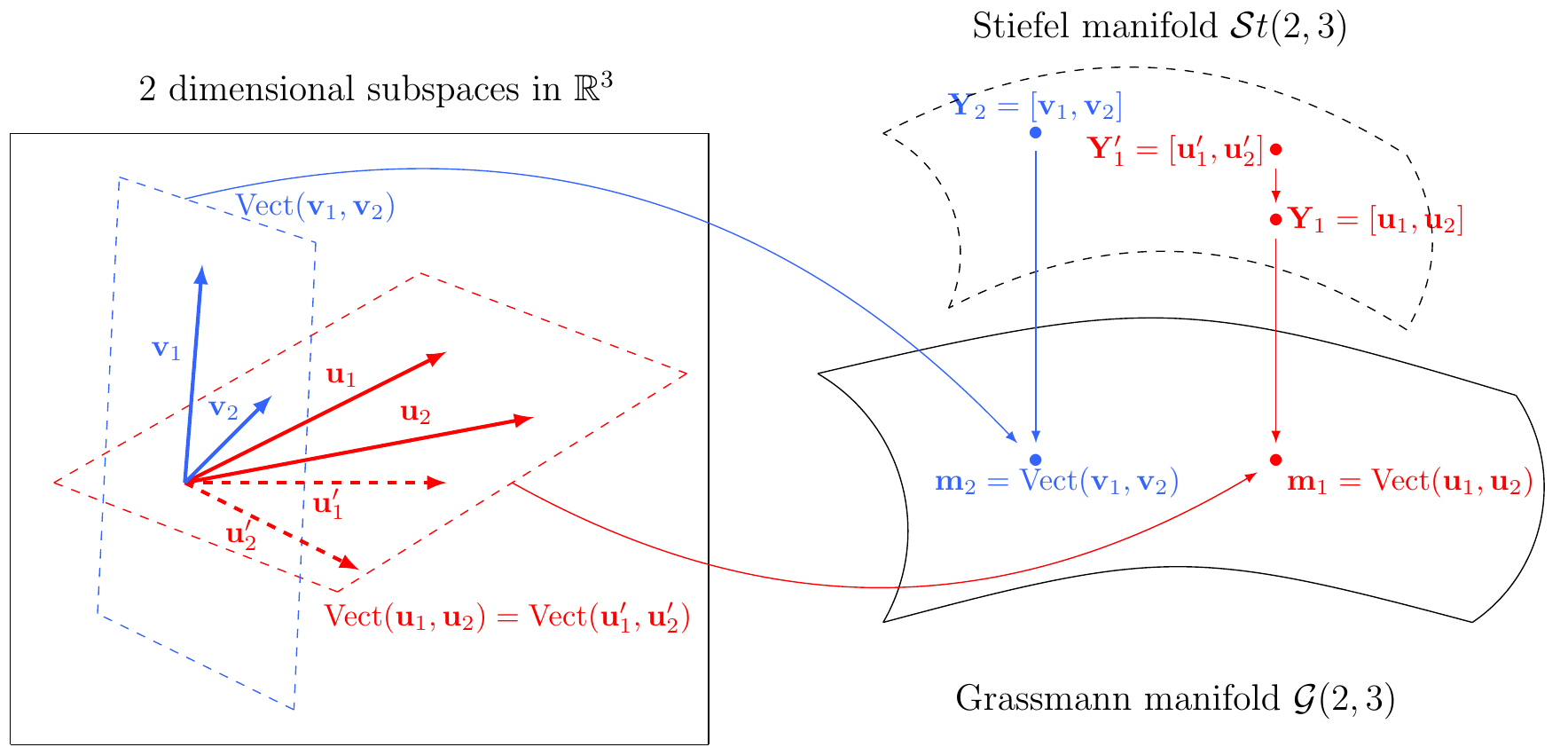}
}
\caption{Points on Stiefel $\mathcal{S}t(2,3)$ and Grassmann manifold $\mathcal{G}(2,3)$.}
\label{fig:Two_Points_Grassmann}  
\end{figure}

\begin{figure}[H]
	\resizebox{0.60\textwidth}{!}{%	
		\includegraphics{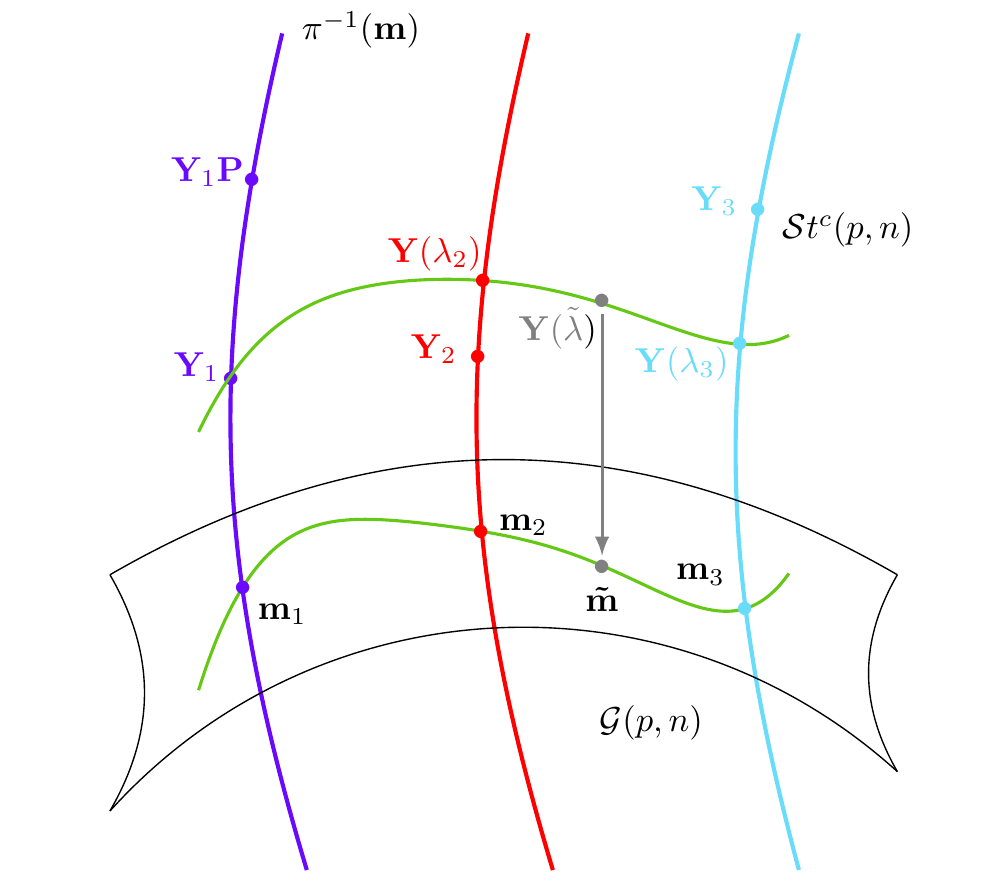}
	}
	\caption{There is a natural projection $\pi: \StieComp{p}{n}\longrightarrow \mathcal{G}(p,n)$ from the compact Stiefel manifold $\StieComp{p}{n}$ to the Grassmannian $\mathcal{G}(p,n)$ of $p$-dimensional subspaces in $\mathbb{R}^n $ which sends a $p$-frame to the subspace spanned by that frame. The fiber over a given point $\bom$ on $\mathcal{G}(p,n)$ is the set of all orthonormal $p$-frames 
		spanning the subspace $\bom$. Computations on $\StieComp{p}{n}$ using the target Algorithm~\ref{alg:Target_Alg} for  $\lambda:=\lambda_{k}$, lead to some matrix $\bY(\lambda_{k})$ generally different from $\bY_{k}$ (except for the reference point), and thus do not produce an interpolation on the points $\bY_{1},\dotsc,\bY_{N}$. }
	\label{fig:Grasmann_Quotient}
\end{figure}

\begin{figure}[H]
\resizebox{0.70\textwidth}{!}{%	
\includegraphics{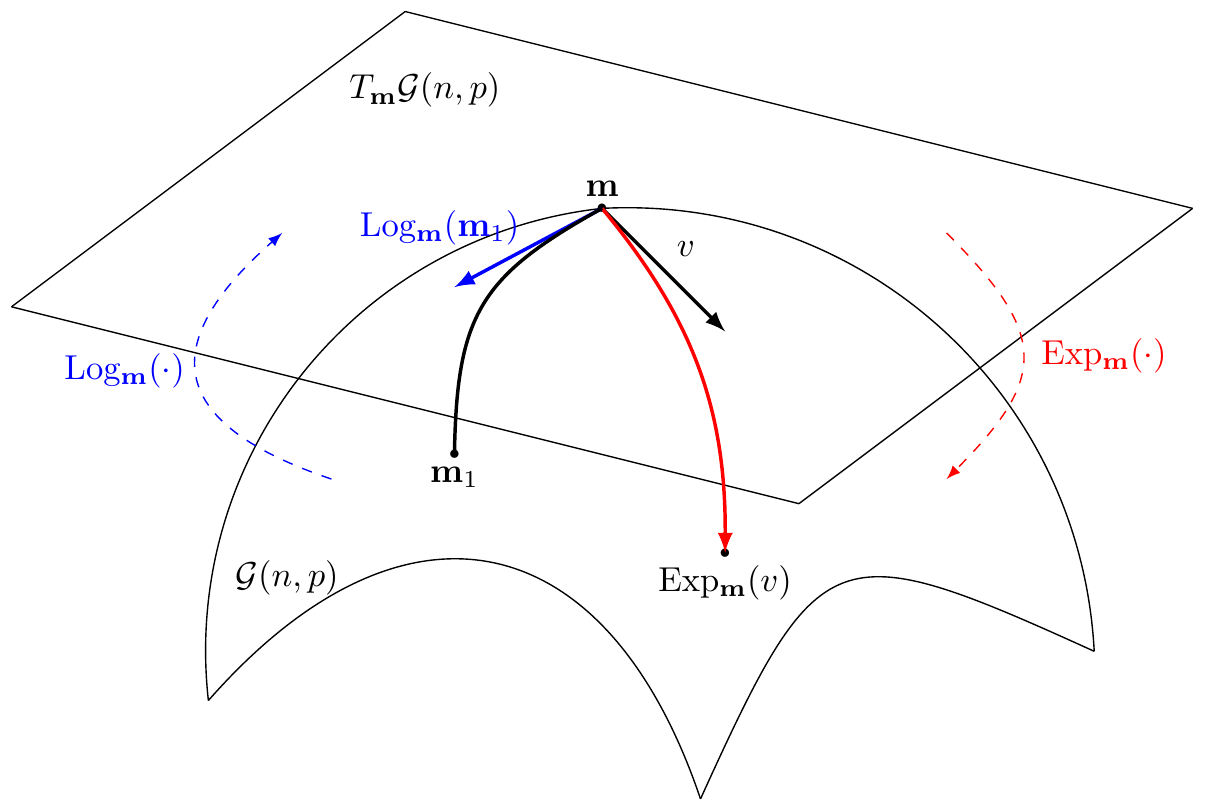}
}
\caption{The exponential  $\Exp_{\bom}$ and the logarithm $\Log_{\bom}$ map on the Grassmann manifold $\mathcal{G}(p,n)$.}
\label{fig:Log_Exp_map}  
\end{figure}

\begin{figure}[H]
\resizebox{0.97\textwidth}{!}{%	
\includegraphics{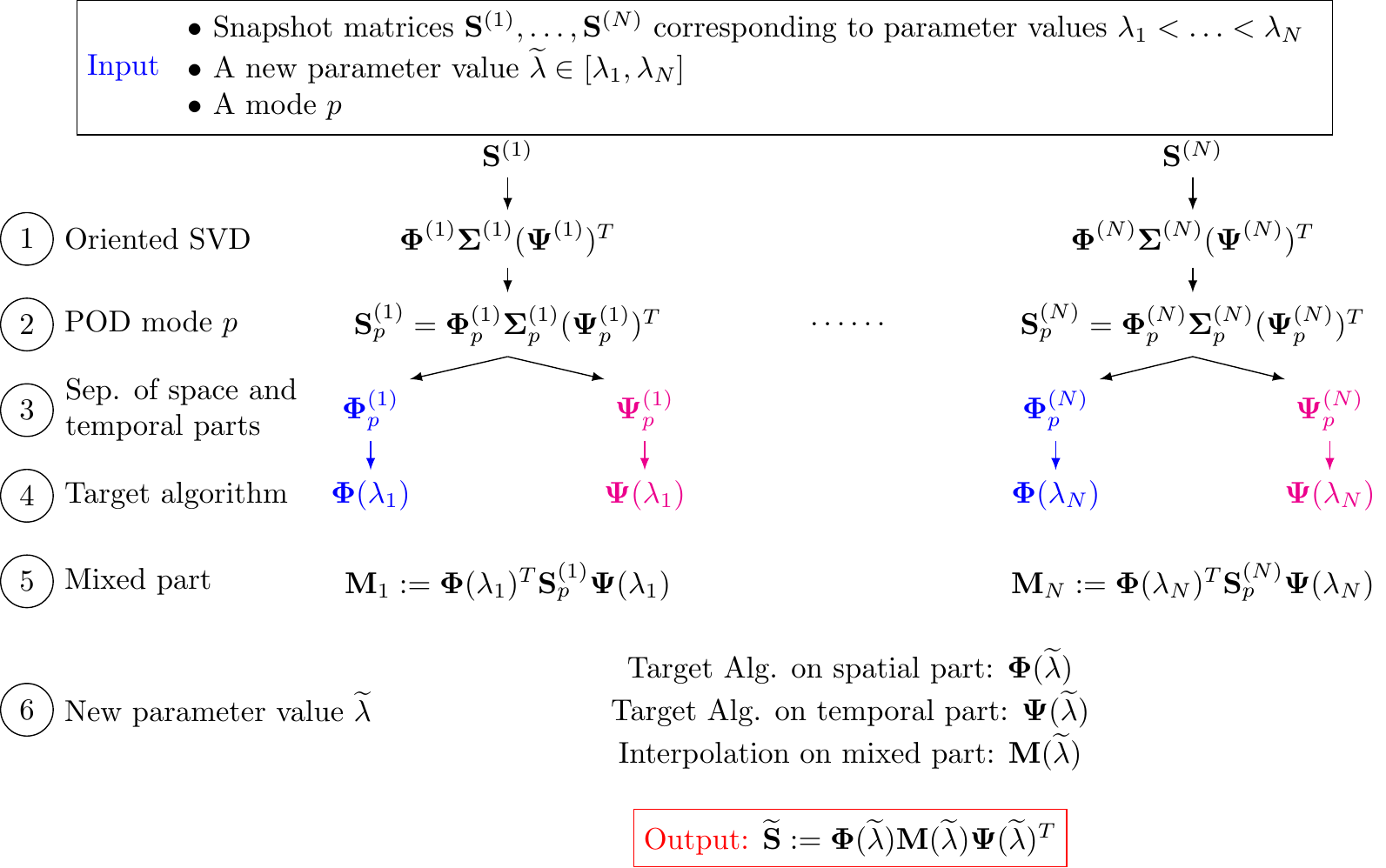}
}
\caption{The Space-Time Algorithm.}
\label{fig:Space_Time_Algo}  
\end{figure}

\begin{figure}[H]
	\begin{subfigure}{.48\textwidth}
		\centering
		% include first image
		\includegraphics[width=\textwidth]{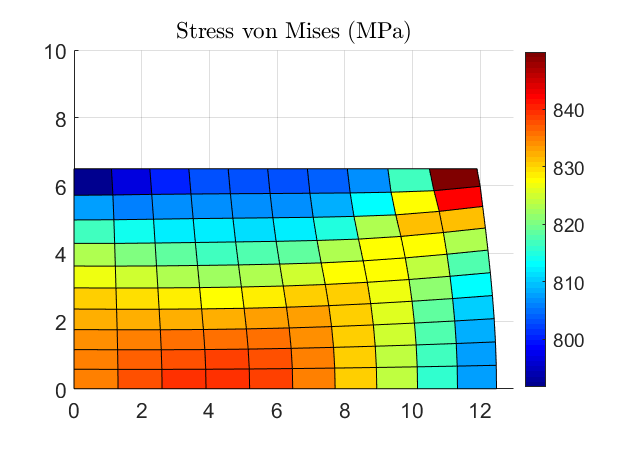}
		\vspace{-2\baselineskip}
		\caption{For $\lambda  = 0.1$}
		\label{fig:m_01}
	\end{subfigure}
	\begin{subfigure}{.48\textwidth}
		\centering
		% include second image
		\includegraphics[width=\textwidth]{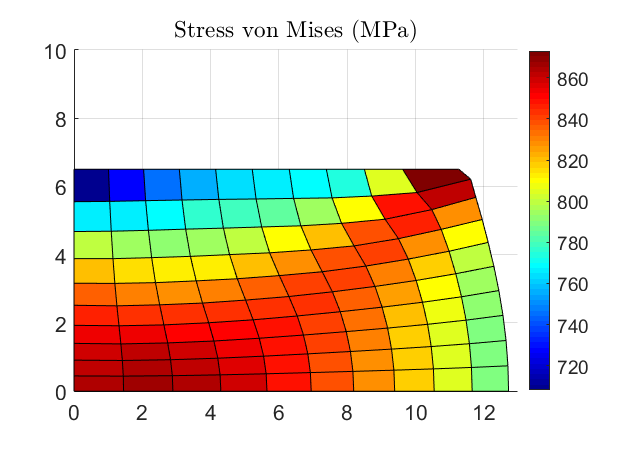}
		\vspace{-2\baselineskip}
		\caption{For $\lambda = 0.3$}
		\label{fig:m_03}
	\end{subfigure}
	\begin{subfigure}{.48\textwidth}
		\centering
		% include second image
		\includegraphics[width=\textwidth]{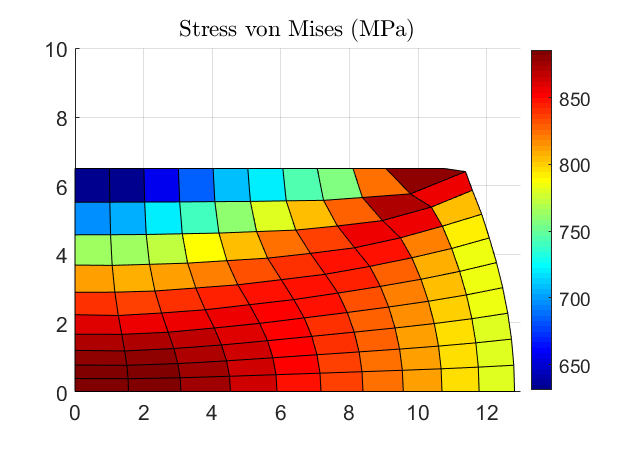}
        \vspace{-2\baselineskip}
		\caption{For $\lambda  = 0.5$}
		\label{fig:m_05}
	\end{subfigure}
	\begin{subfigure}{.48\textwidth}
		\centering
		% include second image
		\includegraphics[width=\textwidth]{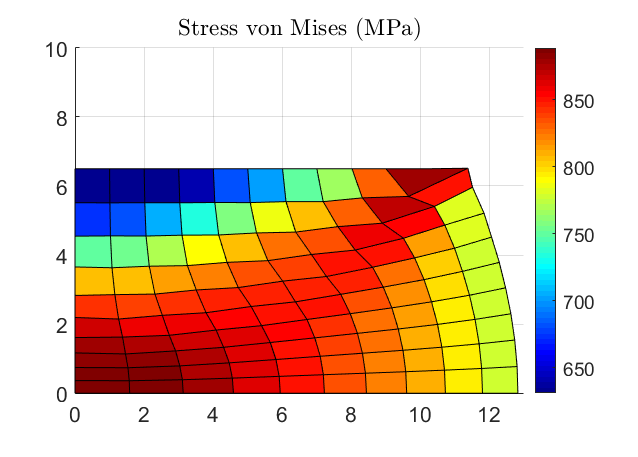}
		\vspace{-2\baselineskip}
		\caption{For $\lambda  = 0.9$}
		\label{fig:m_09}
	\end{subfigure}
	\caption{Deformation patterns of the benchmark metal forming example using different values for the shear friction factor $m$ represented by the parameter $\lambda$.}
	\label{fig:Deformation_patterns}
\end{figure}

\begin{figure}[H]
\resizebox{0.60\textwidth}{!}{%
\includegraphics{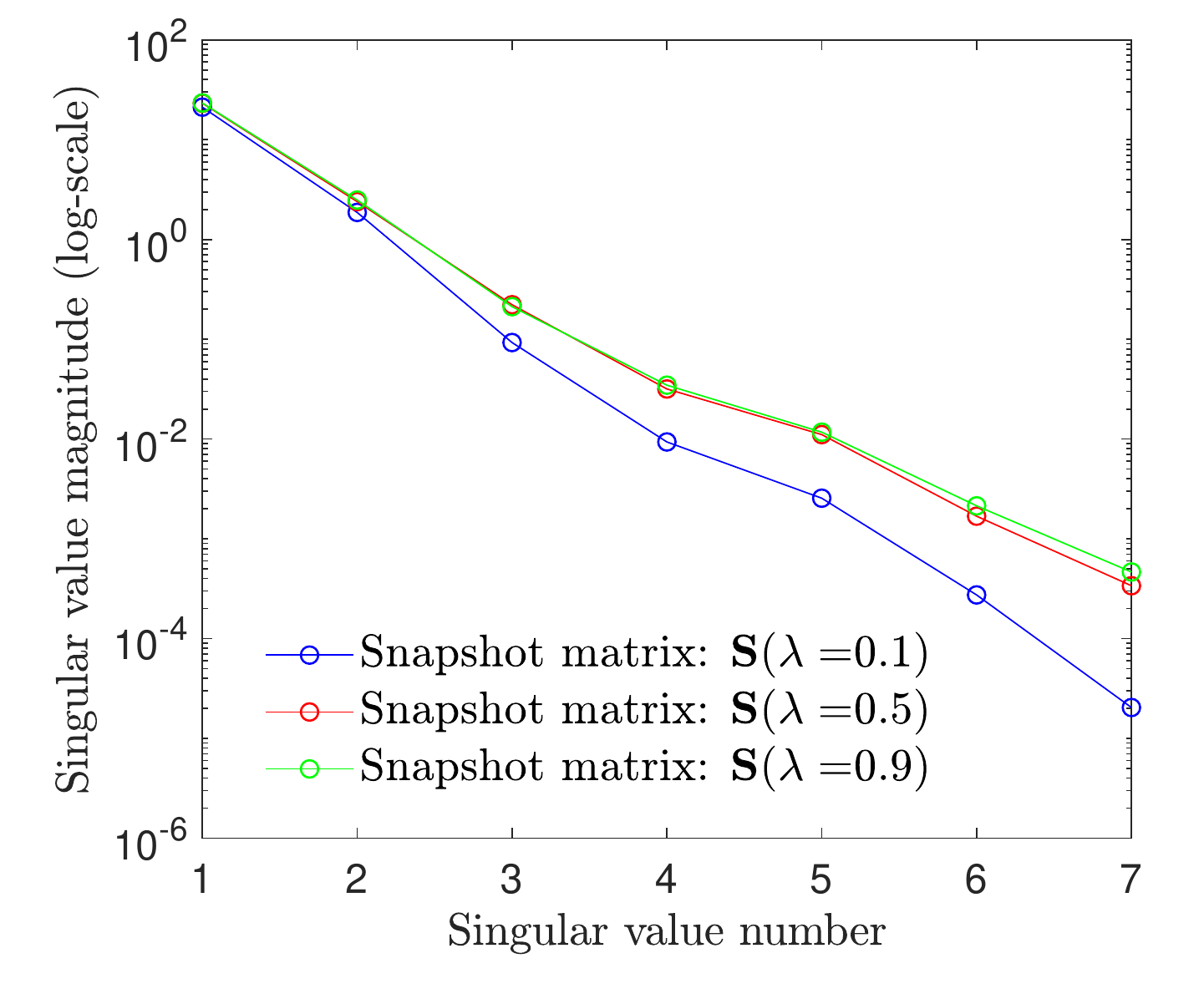}
}
\caption{The eigenvalue spectrum of snapshot matrices $\bS^{(i)}$ corresponding to training points $\lambda \in \Lambda_t = \{0.1, 0.5, 0.9 \}$.}
\label{fig:Sing_values_magnitude}  
\end{figure}

\begin{figure}[H]
\resizebox{0.60\textwidth}{!}{%
\includegraphics{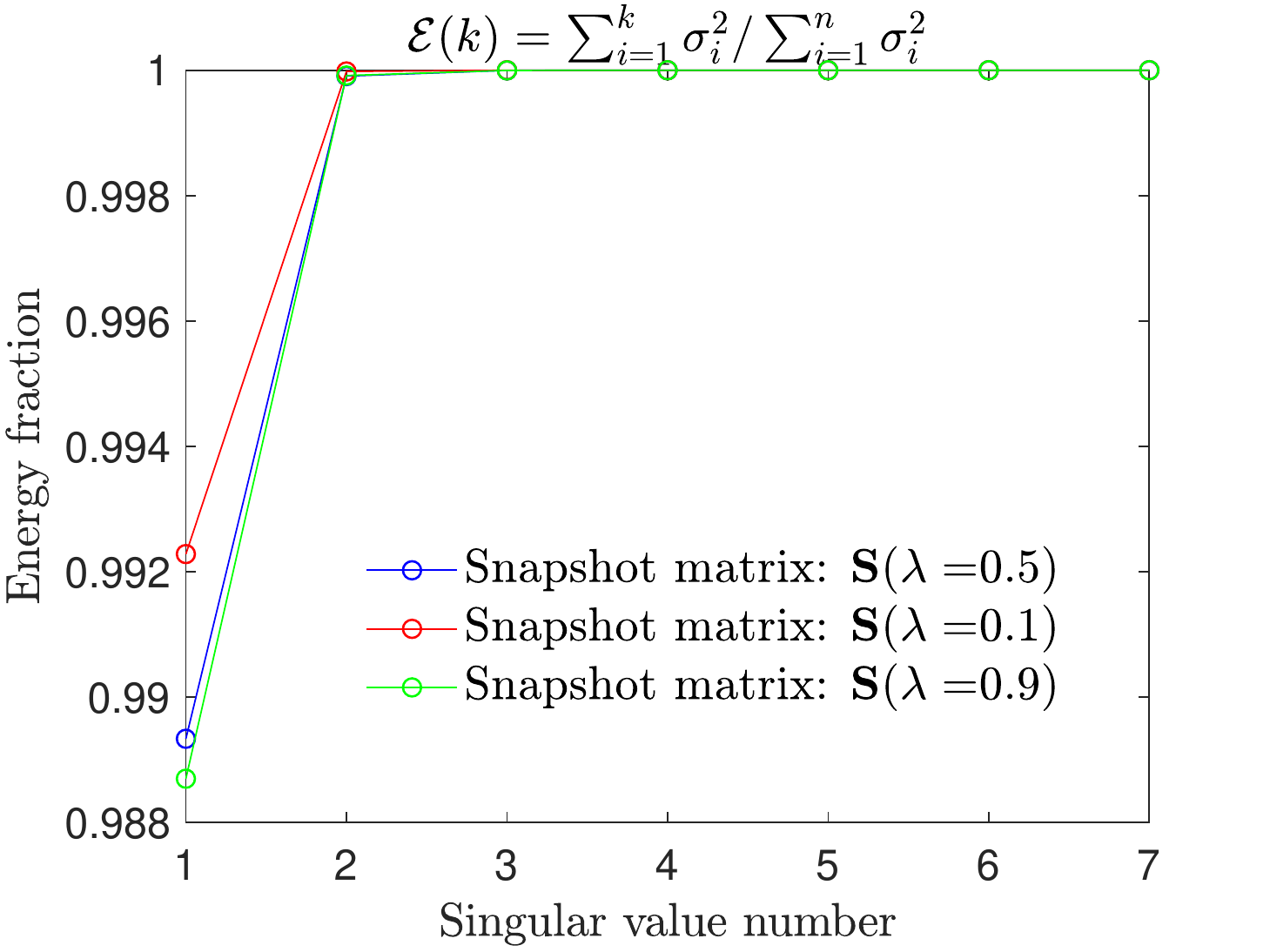}
}
\caption{Energy captured by the singular values of snapshot matrices $\bS^{(i)}$ corresponding to training points $\lambda \in \Lambda_t = \{0.1, 0.5, 0.9 \}$.}
\label{fig:Energy_SVD}  
\end{figure}

\begin{figure}[H]
\resizebox{0.60\textwidth}{!}{%
\includegraphics{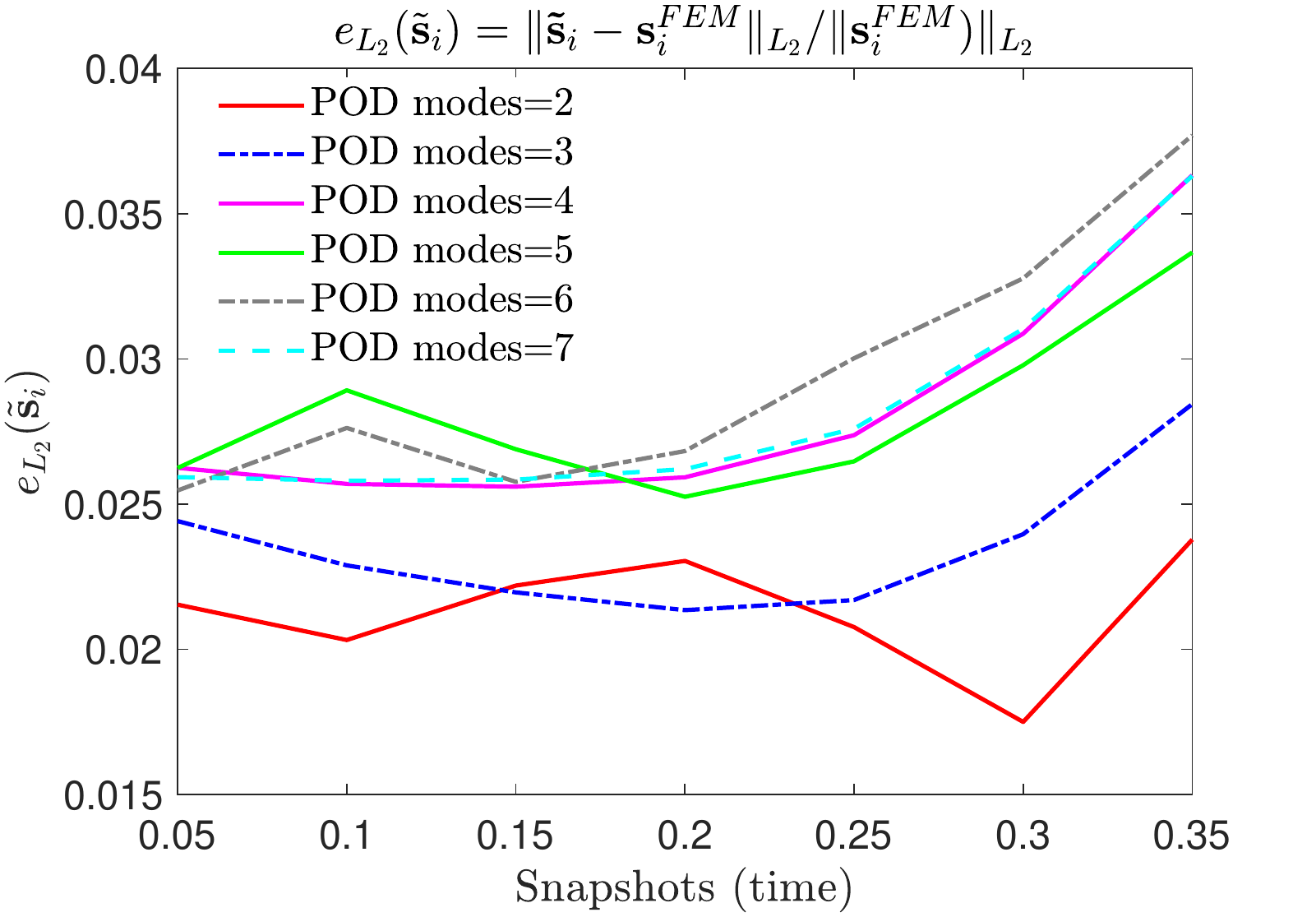}
}
\caption{Performance of ST POD interpolation using the relative $L_2$-error norm  $e_{L_2}(\widetilde{\bs}_i)$ for various  modes $p$; 
 training points $\lambda \in \Lambda_t = \{0.1, 0.5, 0.9 \}$; reference parameter value $\lambda_{i_0}=0.5$;  target point $\widetilde{\lambda}=0.3$.}
\label{fig:L2_norm_POD_modes}
\end{figure}

\begin{figure}[H]
\begin{minipage}[b]{0.48\linewidth}
\centering
\includegraphics[width=\textwidth]{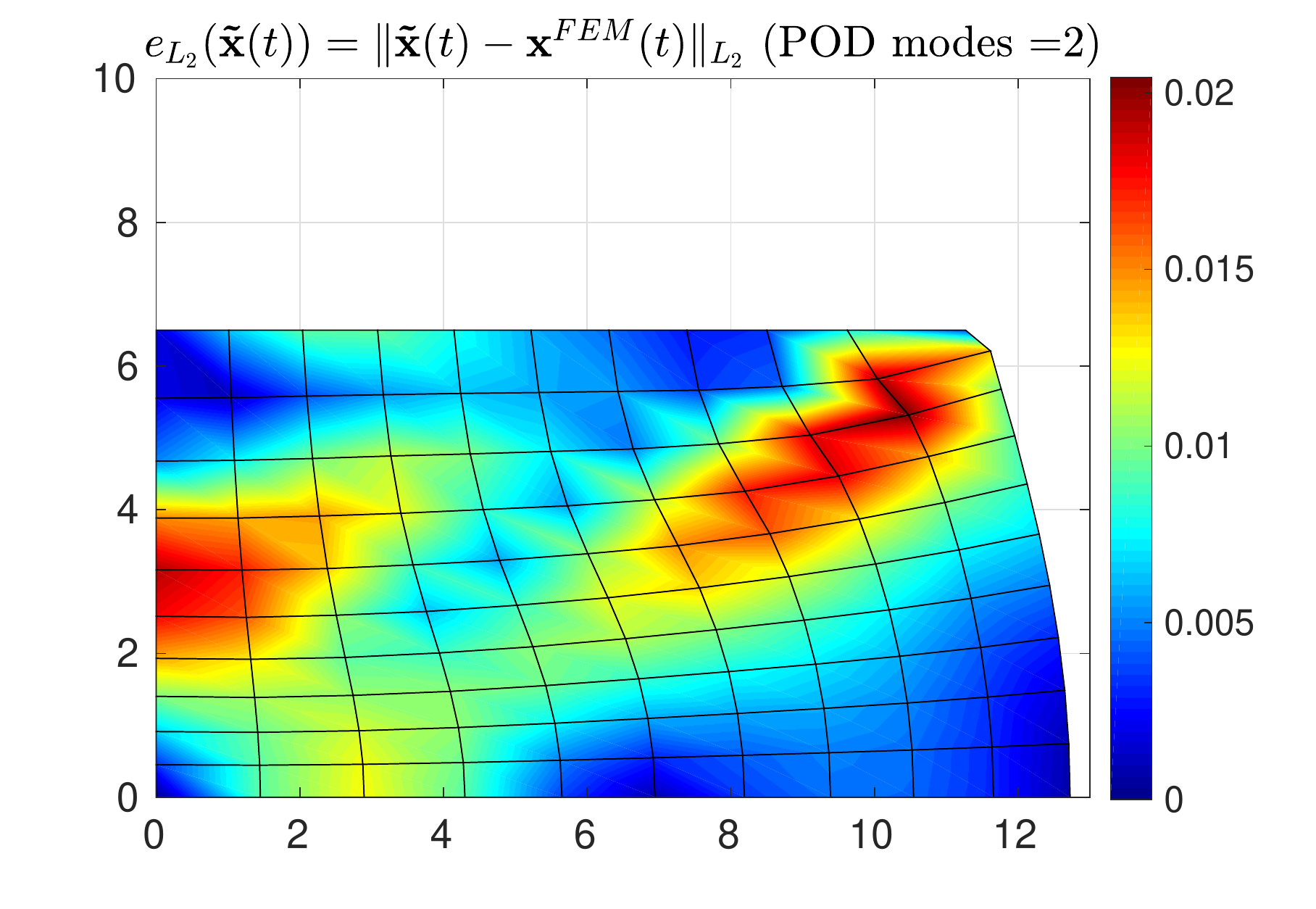}
%\caption{}
%\label{fig:}
\end{minipage}
%\hspace{0.01cm}
\begin{minipage}[b]{0.48\linewidth}
\centering
\includegraphics[width=\textwidth]{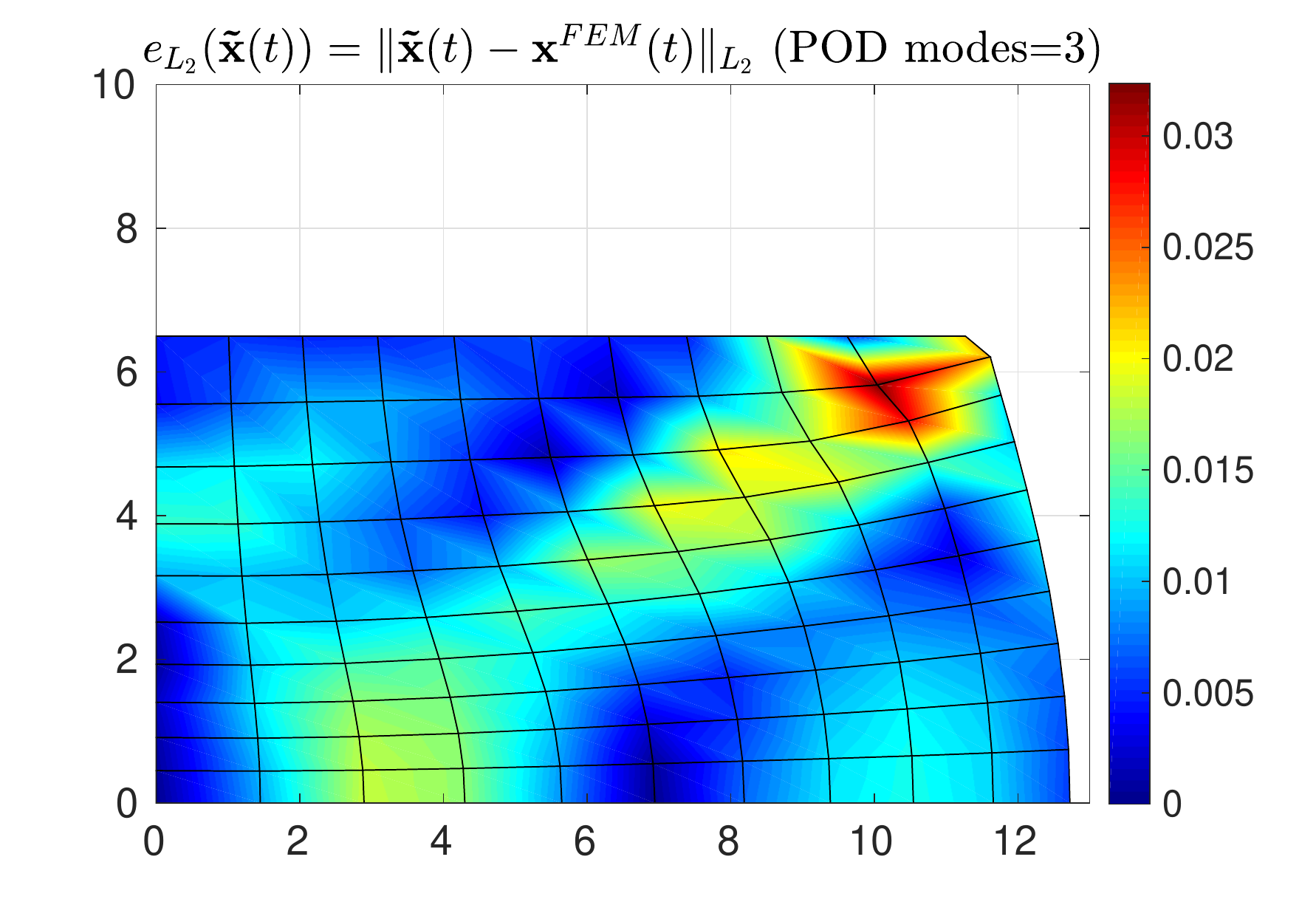}
%\caption{}
%\label{fig:}
\end{minipage}

\begin{minipage}[b]{0.48\linewidth}
\centering
\includegraphics[width=\textwidth]{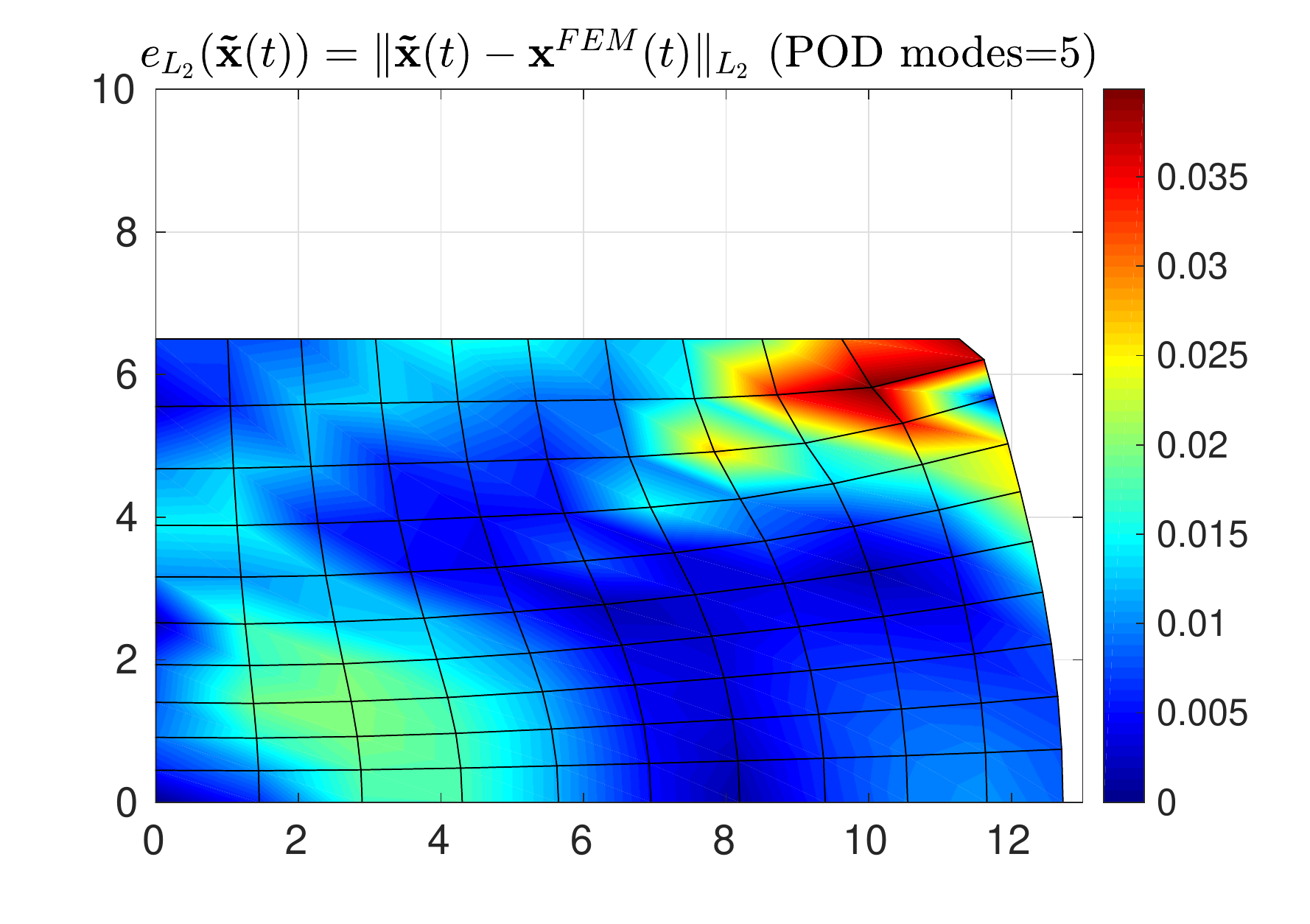}
%\caption{}
%\label{fig:}
\end{minipage}
\begin{minipage}[b]{0.48\linewidth}
\centering
\includegraphics[width=\textwidth]{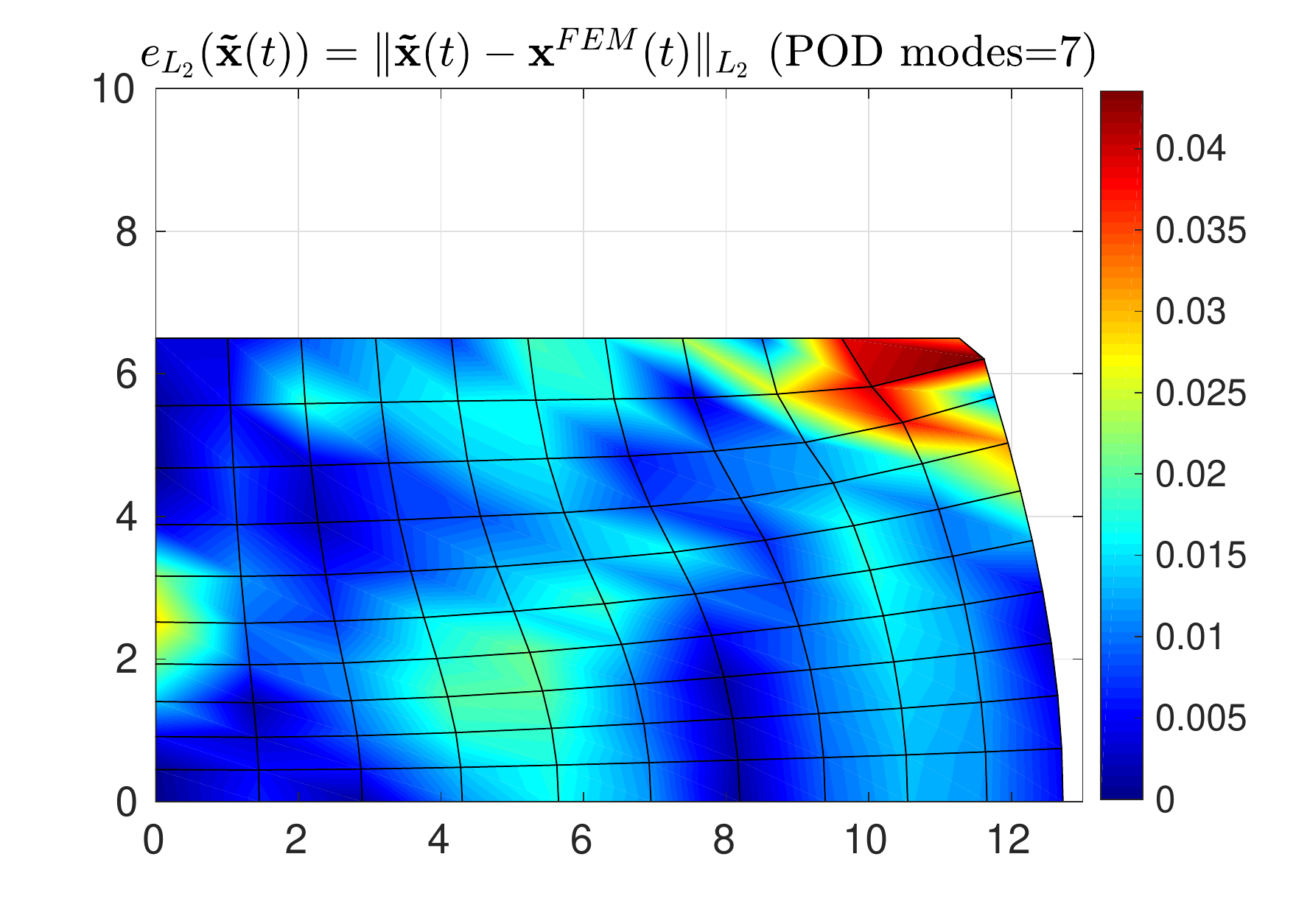}
%\caption{}
%\label{fig:}
\end{minipage}
\caption{The position vector error $e_{L_2}(\mathbf{\widetilde{x}}(t)) = \Vert \mathbf{\widetilde{x}}(t)-\mathbf{x}^{\text{FEM}}(t) \Vert_{L_2}$ of the nodal points at the final deformation state $t=0.35$ s superimposed on the high-fidelity FEM solution; POD modes $p= \{2,3,5,7\}$; training points $\lambda \in \Lambda_t = \{0.1, 0.5, 0.9 \}$; reference parameter value $\lambda_{i_0}=0.5$;  target point $\widetilde{\lambda}=0.3$.}
\label{fig:Snapshot_L2_norm_POD_modes_state_7}
\end{figure}

\begin{figure}[H]
\resizebox{0.60\textwidth}{!}{%
\includegraphics{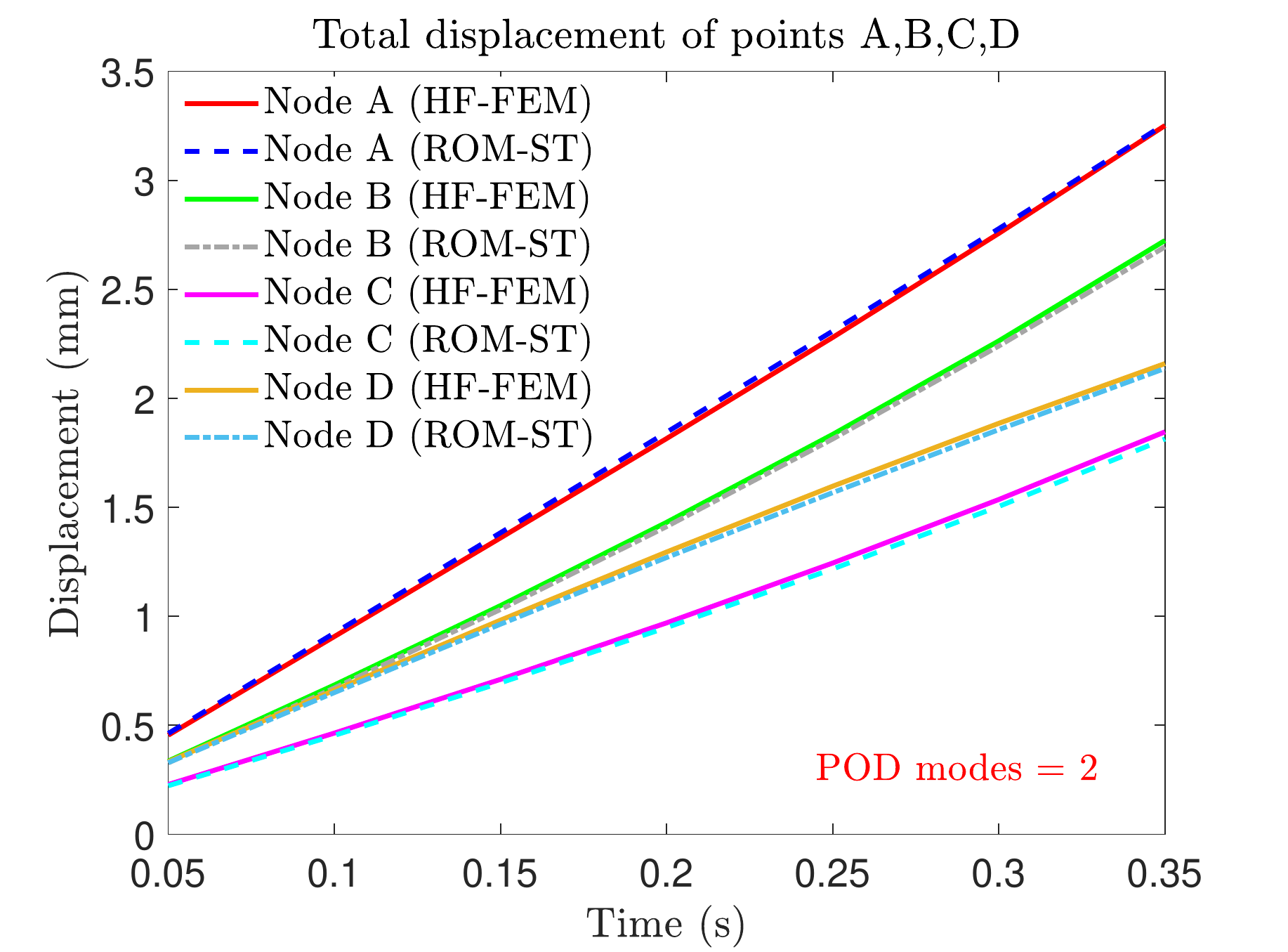}
}
\caption{Comparison of the total displacement of selected nodes against the high-fidelity FEM solution;  training points $\lambda \in \Lambda_t = \{0.1, 0.5, 0.9 \}$;  reference parameter value $\lambda_{i_0}=0.5$;  target point $\widetilde{\lambda}=0.3$; POD modes $p=2$.}
\label{fig:Nodal_time_displacement_histories}
\end{figure}

\begin{figure}[H]
\resizebox{0.60\textwidth}{!}{%
\includegraphics{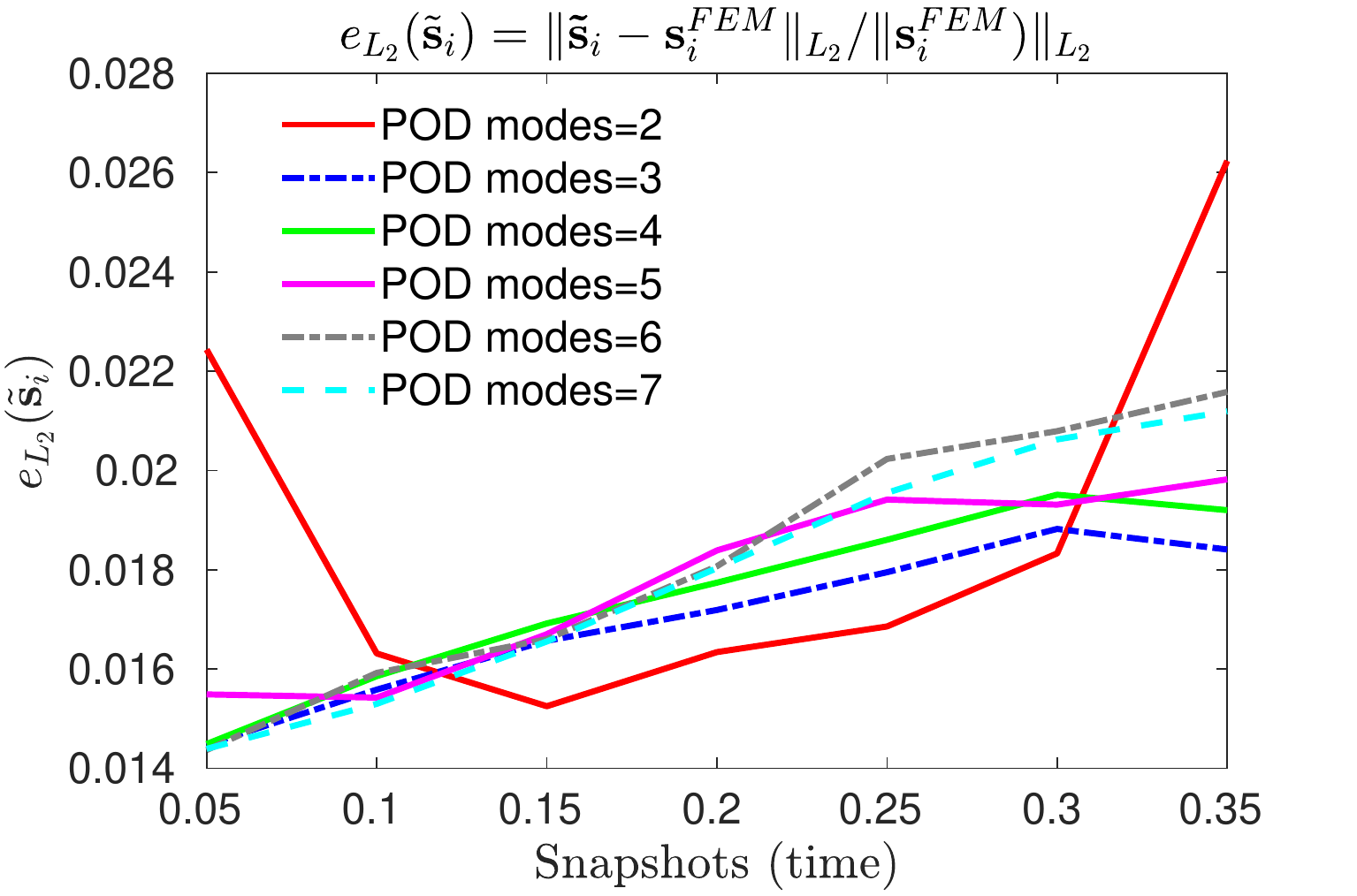}
}
\caption{Performance of ST POD interpolation using the relative $L_2$-error norm  $e_{L_2}(\widetilde{\bs}_i)$ for various POD modes $p$;  
training points $\lambda \in \Lambda_t = \{0.1, 0.5, 0.9 \}$; reference parameter value $\lambda_{i_0}=0.5$;  target point $\widetilde{\lambda}=0.8$.}
\label{fig:L2_norm_POD_modes_new_interp_target_m_08}
\end{figure}

\begin{figure}[H]
\begin{subfigure}{.48\textwidth}
\centering
% include first image
\includegraphics[width=\textwidth]{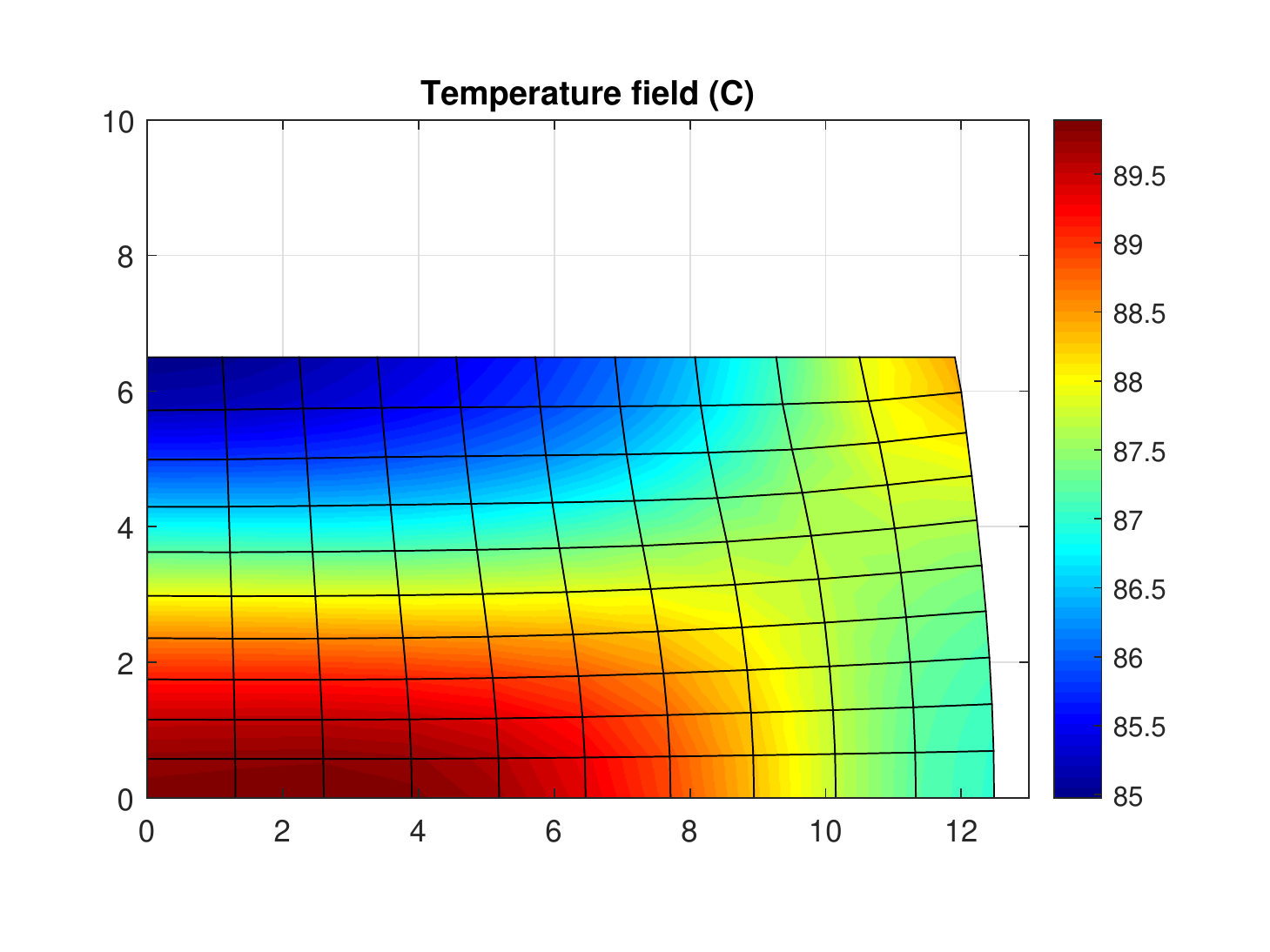}
\vspace{-2\baselineskip}
\caption{For $m = 0.1$}
\label{fig:friction_01}
\end{subfigure}
\begin{subfigure}{.48\textwidth}
\centering
% include second image
\includegraphics[width=\textwidth]{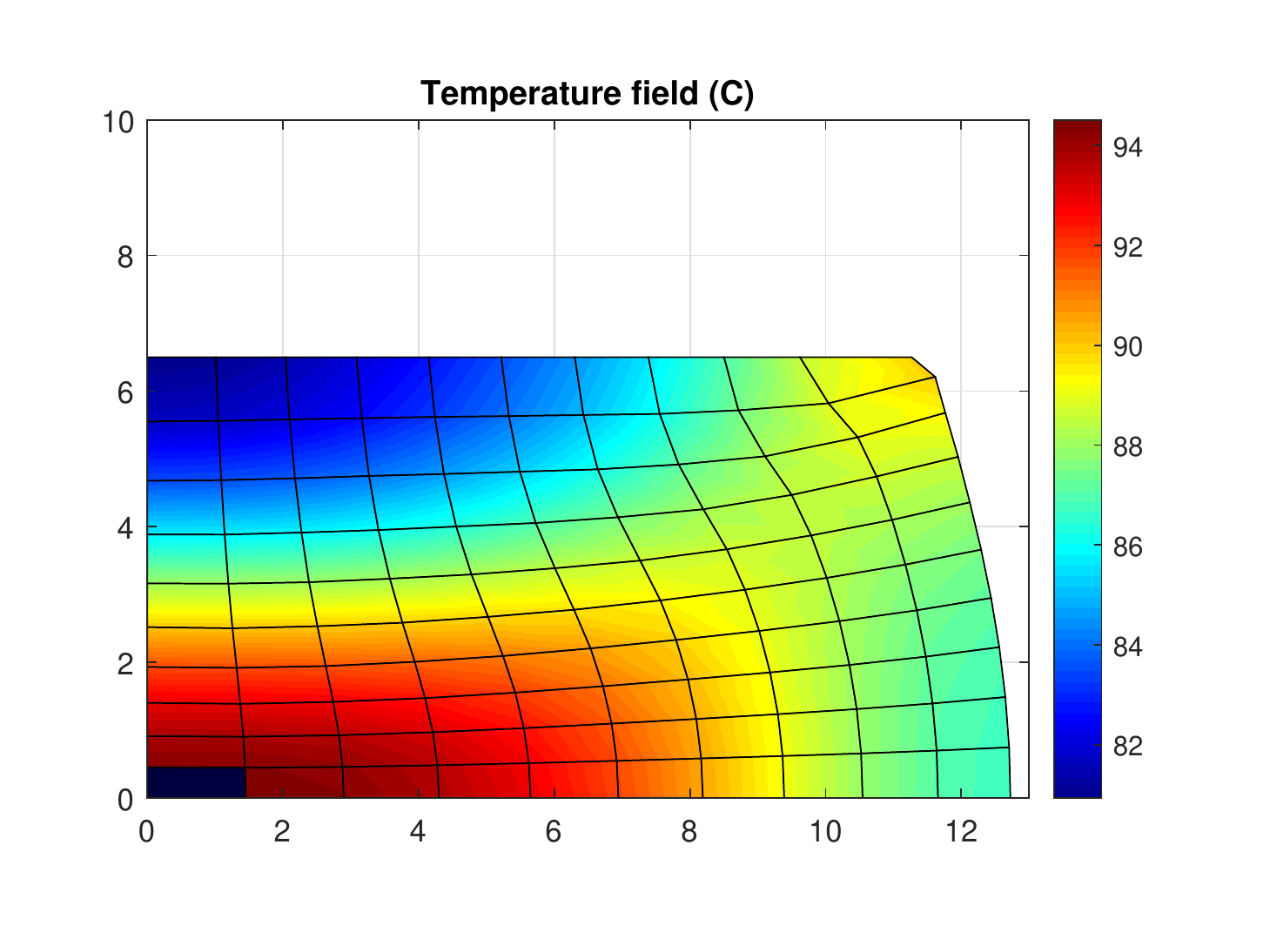}
\vspace{-2\baselineskip}
\caption{For $m = 0.3$}
\label{fig:friction_03}
\end{subfigure}
\begin{subfigure}{.48\textwidth}
\centering
% include second image
\includegraphics[width=\textwidth]{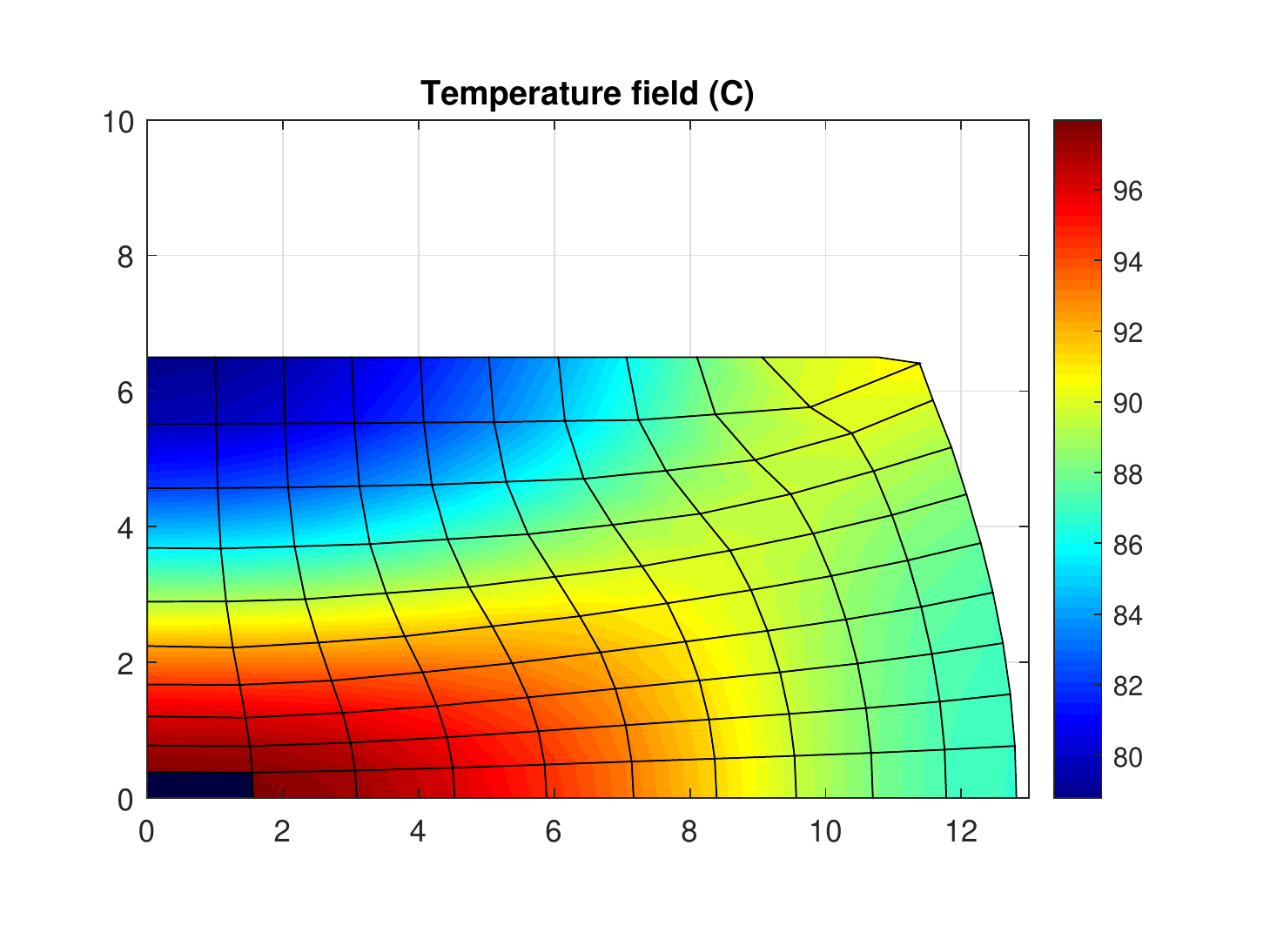}
\vspace{-2\baselineskip}
\caption{For $m = 0.5$}
\label{fig:friction_05}
\end{subfigure}
\begin{subfigure}{.48\textwidth}
\centering
% include second image
\includegraphics[width=\textwidth]{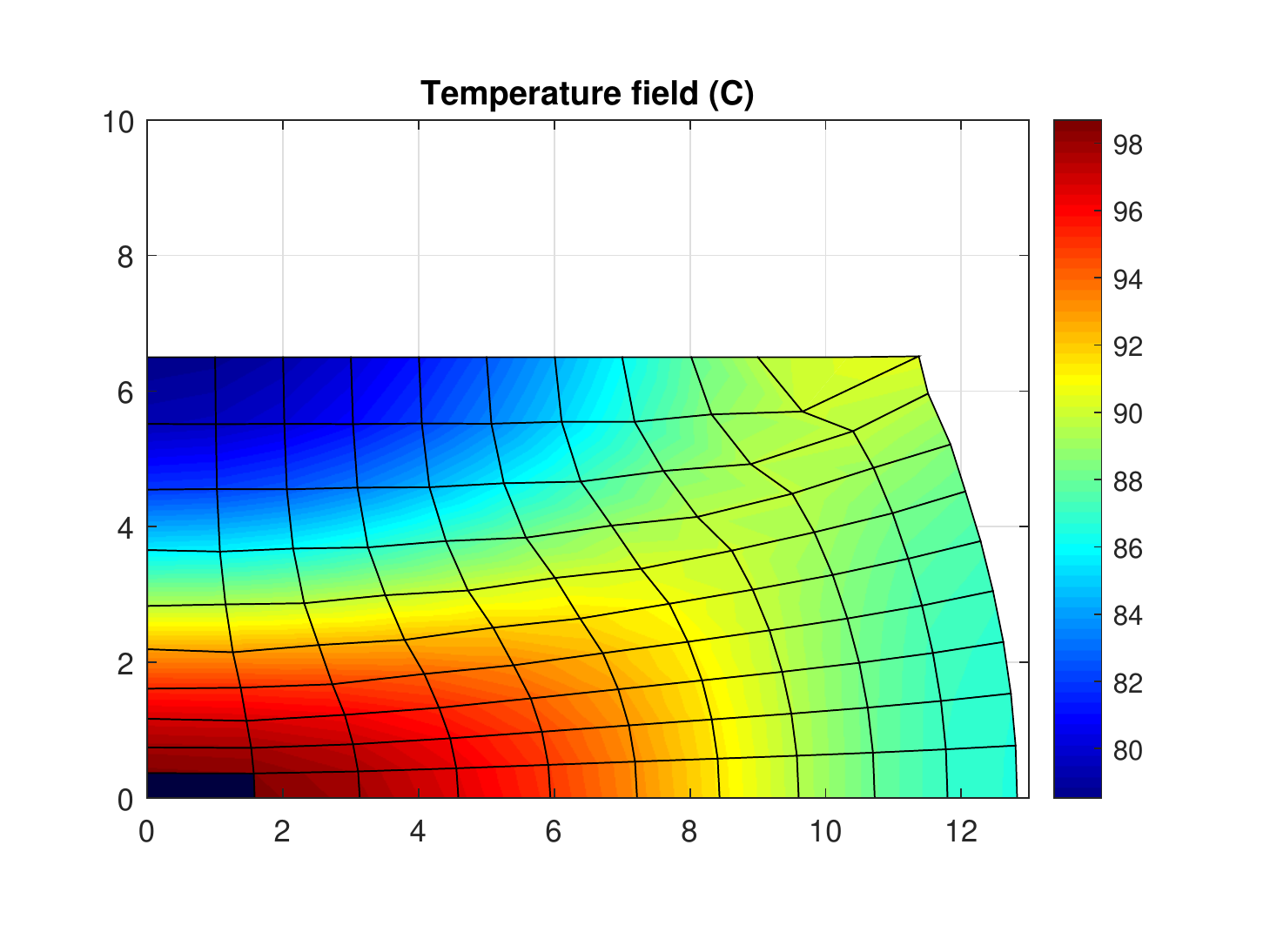}
\vspace{-2\baselineskip}
\caption{For $m = 0.9$}
\label{fig:friction_09}
\end{subfigure}
\caption{Temperature profiles at the final compression state $t=0.35$ s obtained using different values of the shear friction factor $m$ represented by parameter $\lambda$.}
\label{fig:Temperature_profiles}
\end{figure}

\begin{figure}[H]
\resizebox{0.60\textwidth}{!}{%
\includegraphics{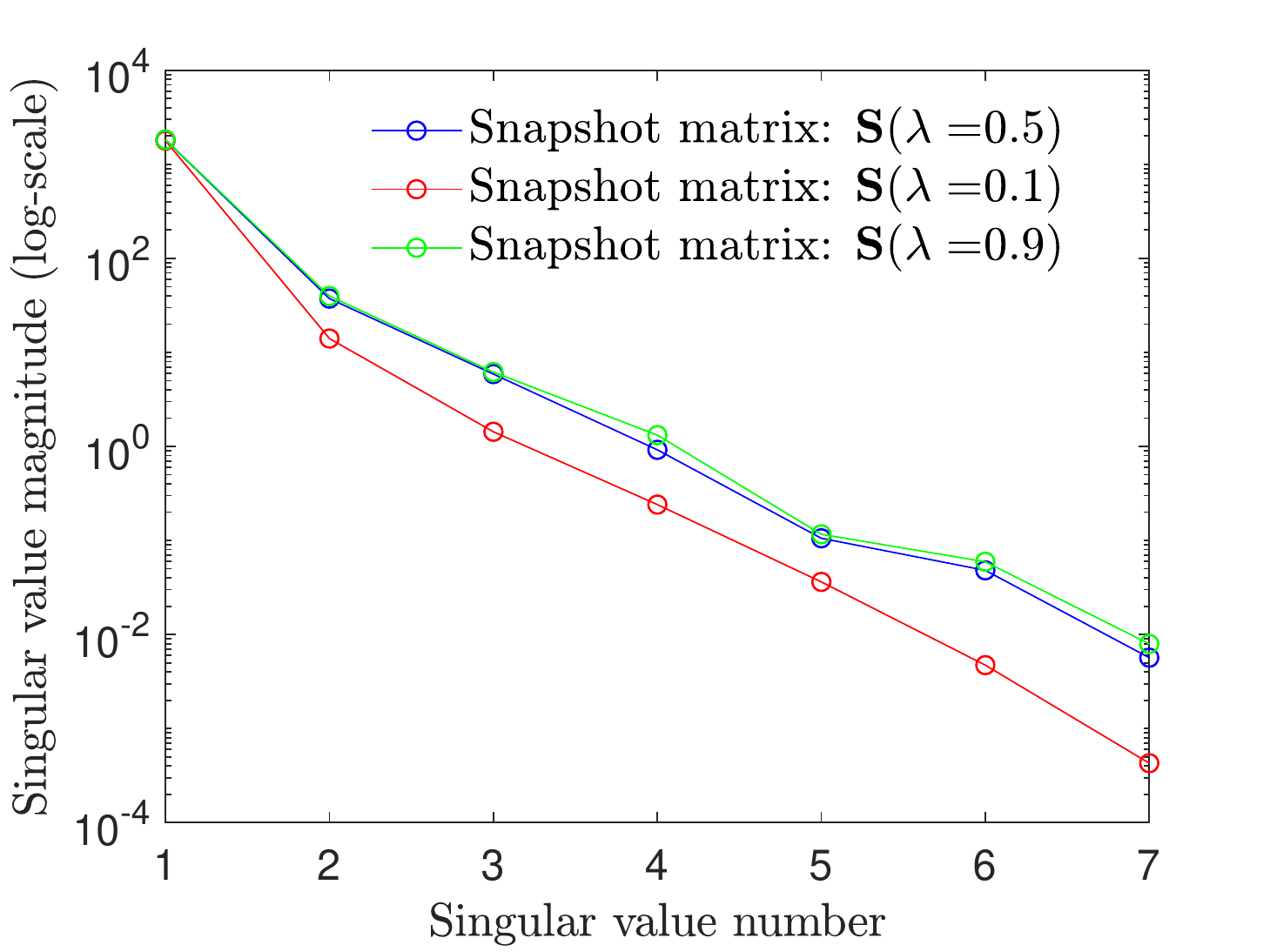}
}
\caption{The eigenvalue spectrum of snapshot matrices $\bS^{(i)}$ corresponding to training points $\lambda \in \Lambda_t = \{0.1, 0.5, 0.9 \}$.}
\label{fig:Sing_values_magnitude_temperature}    
\end{figure}

\begin{figure}[H]
\resizebox{0.60\textwidth}{!}{%
\includegraphics{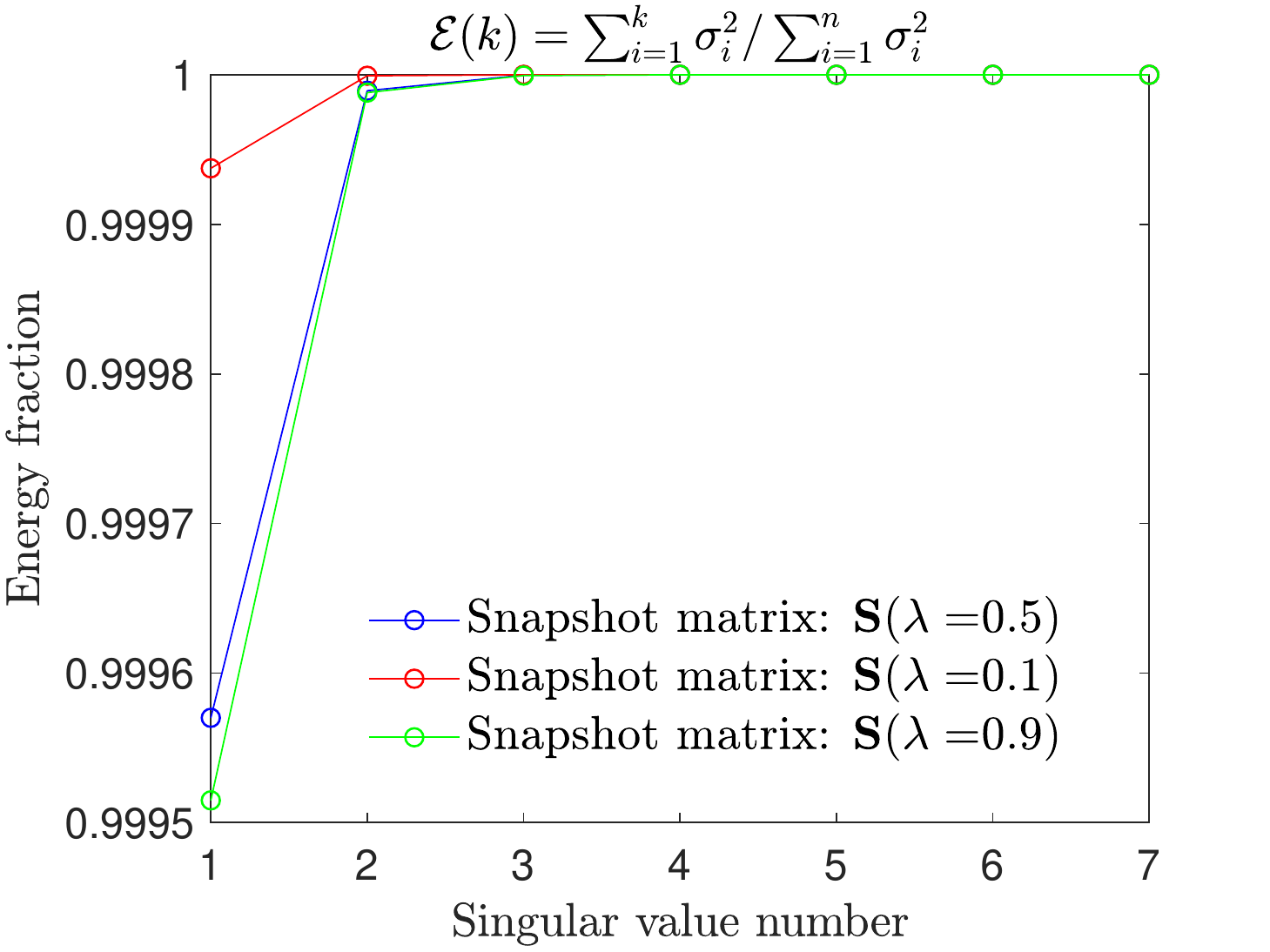}
}
\caption{Energy captured by the singular values of snapshot matrices $\bS^{(i)}$ corresponding to training points $\lambda \in \Lambda_t = \{0.1, 0.5, 0.9 \}$.}
\label{fig:Energy_SVD_temperature}  
\end{figure}

\begin{figure}[H]
\resizebox{0.60\textwidth}{!}{%
\includegraphics{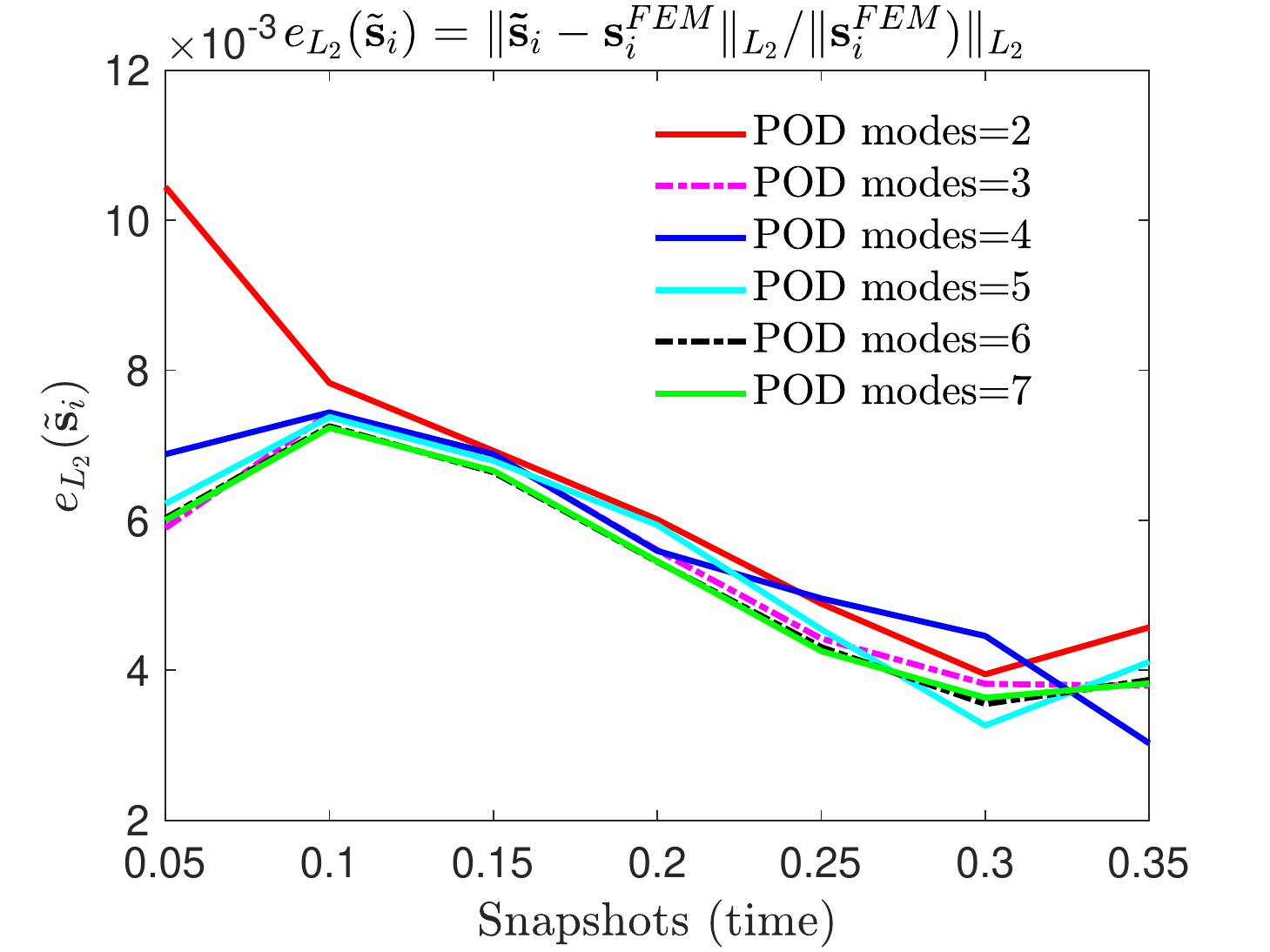}
}
\caption{Performance of ST POD interpolation using the relative $L_2$-error norm  $e_{L_2}(\widetilde{\bs}_i)$ for various POD modes $p$; training points $\lambda \in \Lambda_t = \{0.1, 0.5, 0.9 \}$; reference parameter value $\lambda_{i_0}=0.5$;  target point $\widetilde{\lambda}=0.3$.}
\label{fig:L2_norm_POD_modes_temperature}
\end{figure}

\begin{figure}[H]
\resizebox{0.60\textwidth}{!}{%
\includegraphics{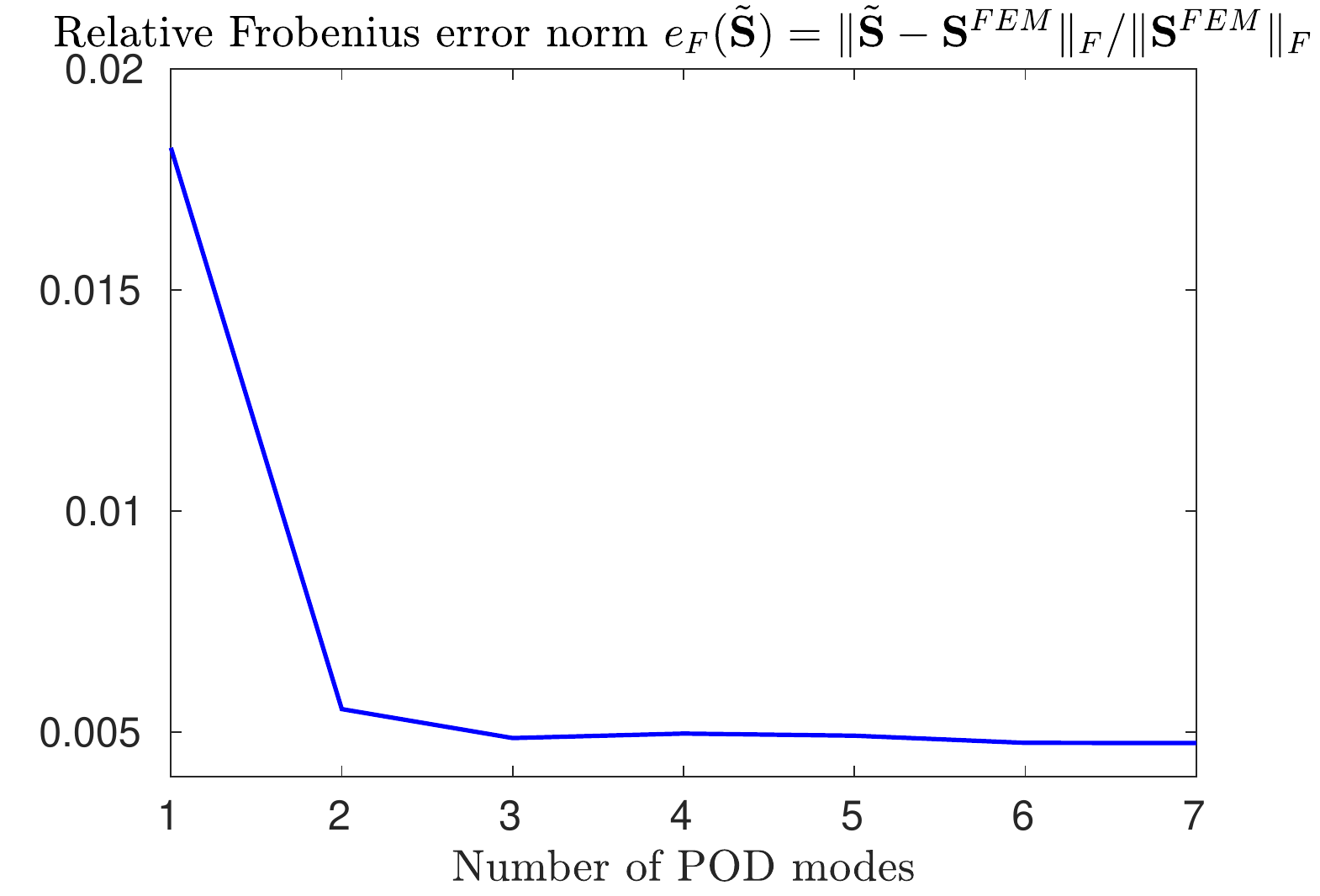}
}
\caption{Performance of the POD interpolation using the relative Frobenius error norm $e_{F}(\widetilde{\bS})$ against the number of POD modes $p$; training points $\lambda \in \Lambda_t = \{0.1, 0.5, 0.9 \}$; reference parameter value $\lambda_{i_0}=0.5$;  target point $\widetilde{\lambda}=0.3$.}
\label{fig:Frobenius_error_norm_temperature}
\end{figure}

\begin{figure}[H]
\resizebox{0.7\textwidth}{!}{%
\includegraphics{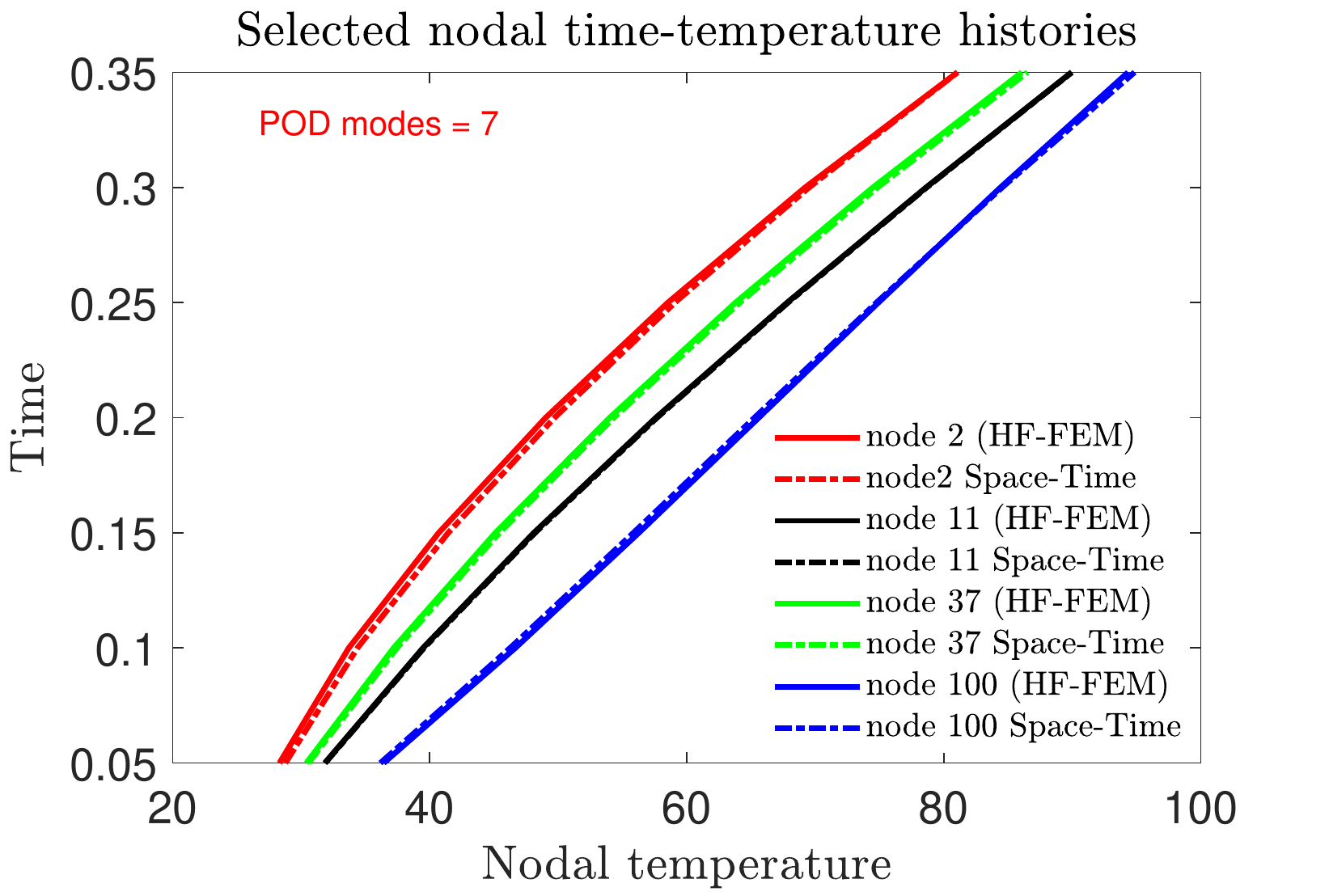}
}
\caption{Temperature evolution of selected nodal points validated against the high-fidelity FEM solution; ST POD and HF-FEM solutions virtually coincide;  training points $\lambda \in \Lambda_t = \{0.1, 0.5, 0.9 \}$; reference parameter value $\lambda_{i_0}=0.5$; target point $\widetilde{\lambda}=0.3$; POD modes $p=7$.}
\label{fig:Nodal_time_temperature_histories}
\end{figure}

\pagebreak

% BibTeX users please use
% \bibliographystyle{}
% \bibliography{}
%

\end{document}